\documentclass[11pt]{amsart}

\usepackage[all]{xy}
\usepackage{mathrsfs}
\usepackage{amsmath, amsthm, amssymb, amscd, amsgen,amsxtra,amsfonts, amsbsy}
\allowdisplaybreaks
\usepackage[all]{xy}
\usepackage{verbatim}

\usepackage{tikz}
\usepackage{pgfplots}
\usepackage{tikz-cd}
\usepackage{etex}
\usepackage{calligra,mathrsfs}
\usepackage{color}
\definecolor{shadecolor}{gray}{0.875}
\definecolor{dblue}{rgb}{0,0,.6}
\usepackage[colorlinks=true, linkcolor=dblue, citecolor=dblue, filecolor = dblue, menucolor = dblue, urlcolor = dblue]{hyperref}
\usepackage{mathtools}
\usepackage{extarrows}
\usepackage{ifthen}
\usepackage{bm}

\usepackage{graphicx}
\usepackage{caption,rotating}
\usepackage{color}
\usepackage{subcaption}

\usepackage{enumitem}

\newcommand{\mathds}[1]{{\mathbb #1}}

\numberwithin{equation}{section}

\usepackage[toc,page]{appendix}

\begin{document}
%
%
%
\theoremstyle{definition}
\newtheorem{Definition}{Definition}[section]
\newtheorem*{Definitionx}{Definition}
\newtheorem{Convention}{Definition}[section]
\newtheorem{Construction}[Definition]{Construction}
\newtheorem{Example}[Definition]{Example}
\newtheorem{Exercise}[Definition]{Exercise}
\newtheorem{Examples}[Definition]{Examples}
\newtheorem{Remark}[Definition]{Remark}
\newtheorem*{Remarkx}{Remark}
\newtheorem{Remarks}[Definition]{Remarks}
\newtheorem{Caution}[Definition]{Caution}
\newtheorem{Conjecture}[Definition]{Conjecture}
\newtheorem*{Conjecturex}{Conjecture}
\newtheorem{Question}[Definition]{Question}
\newtheorem*{Questionx}{Question}
\newtheorem*{Acknowledgements}{Acknowledgements}
\newtheorem*{Notation}{Notation}
\newtheorem*{Organization}{Organization}
\newtheorem*{Disclaimer}{Disclaimer}
\theoremstyle{plain}
\newtheorem{Theorem}[Definition]{Theorem}
\newtheorem*{Theoremx}{Theorem}

\newtheorem{Proposition}[Definition]{Proposition}
\newtheorem*{Propositionx}{Proposition}
\newtheorem{Lemma}[Definition]{Lemma}
\newtheorem{Corollary}[Definition]{Corollary}
\newtheorem*{Corollaryx}{Corollary}
\newtheorem{Fact}[Definition]{Fact}
\newtheorem{Facts}[Definition]{Facts}
\newtheorem{Claim}[Definition]{Claim}
\newtheoremstyle{voiditstyle}{3pt}{3pt}{\itshape}{\parindent}%
{\bfseries}{.}{ }{\thmnote{#3}}%
\theoremstyle{voiditstyle}
\newtheorem*{VoidItalic}{}
\newtheoremstyle{voidromstyle}{3pt}{3pt}{\rm}{\parindent}%
{\bfseries}{.}{ }{\thmnote{#3}}%
\theoremstyle{voidromstyle}
\newtheorem*{VoidRoman}{}

\newenvironment{specialproof}[1][\proofname]{\noindent\textit{#1.} }{\qed\medskip}
\newcommand{\blowup}{\rule[-3mm]{0mm}{0mm}}
\newcommand{\cal}{\mathcal}
\newcommand{\Aff}{{\mathds{A}}}
\newcommand{\BB}{{\mathds{B}}}
\newcommand{\BC}{{\mathds{C}}}
\newcommand{\CC}{{\mathds{C}}}
\newcommand{\EE}{{\mathds{E}}}
\newcommand{\FF}{{\mathds{F}}}
\newcommand{\GG}{{\mathds{G}}}
\newcommand{\HH}{{\mathds{H}}}
\newcommand{\NN}{{\mathds{N}}}
\newcommand{\ZZ}{{\mathds{Z}}}
\newcommand{\PP}{{\mathds{P}}}
\newcommand{\QQ}{{\mathds{Q}}}
\newcommand{\RR}{{\mathds{R}}}
\newcommand{\BA}{{\mathds{A}}}
\newcommand{\D}{{\mathcal D}}
\newcommand{\LL}{{\mathcal L}}
\newcommand{\VV}{{\mathcal V}}
\newcommand{\WW}{{\mathcal W}}
\newcommand{\Liea}{{\mathfrak a}}
\newcommand{\Lieb}{{\mathfrak b}}
\newcommand{\Lieg}{{\mathfrak g}}
\newcommand{\Liem}{{\mathfrak m}}
\newcommand{\ideala}{{\mathfrak a}}
\newcommand{\idealb}{{\mathfrak b}}
\newcommand{\idealg}{{\mathfrak g}}
\newcommand{\idealm}{{\mathfrak m}}
\newcommand{\idealp}{{\mathfrak p}}
\newcommand{\idealq}{{\mathfrak q}}
\newcommand{\idealI}{{\cal I}}
\newcommand{\lin}{\sim}
\newcommand{\num}{\equiv}
\newcommand{\dual}{\ast}
\newcommand{\iso}{\cong}
\newcommand{\homeo}{\approx}
\newcommand{\mm}{{\mathfrak m}}
\newcommand{\pp}{{\mathfrak p}}
\newcommand{\qq}{{\mathfrak q}}
\newcommand{\rr}{{\mathfrak r}}
\newcommand{\pP}{{\mathfrak P}}
\newcommand{\qQ}{{\mathfrak Q}}
\newcommand{\rR}{{\mathfrak R}}
\newcommand{\OO}{{\cal O}}
\newcommand{\CO}{{\mathcal O}}
\newcommand{\numero}{{n$^{\rm o}\:$}}
\newcommand{\mf}[1]{\mathfrak{#1}}
\newcommand{\mc}[1]{\mathcal{#1}}
\newcommand{\into}{{\hookrightarrow}}
\newcommand{\onto}{{\twoheadrightarrow}}
\newcommand{\Spec}{{\rm Spec}\:}
\newcommand{\BigSpec}{{\rm\bf Spec}\:}
\newcommand{\Spf}{{\rm Spf}\:}
\newcommand{\Proj}{{\rm Proj}\:}
\newcommand{\Pic}{{\rm Pic }}
\newcommand{\MW}{{\rm MW }}
\newcommand{\Br}{{\rm Br}}
\newcommand{\NS}{{\rm NS}}
\newcommand{\Sym}{{\mathfrak S}}
\newcommand{\Aut}{{\rm Aut}}
\newcommand{\Autp}{{\rm Aut}^p}
\newcommand{\ord}{{\rm ord}}
\newcommand{\coker}{{\rm coker}\,}
\newcommand{\divisor}{{\rm div}}
\newcommand{\Def}{{\rm Def}}
\newcommand{\rank}{\mathop{\mathrm{rank}}\nolimits}
\newcommand{\Ext}{\mathop{\mathrm{Ext}}\nolimits}
\newcommand{\EXT}{\mathop{\mathscr{E}{\kern -2pt {xt}}}\nolimits}
\newcommand{\Hom}{\mathop{\mathrm{Hom}}\nolimits}
\newcommand{\Bk}{\mathop{\mathrm{Bk}}\nolimits}
\newcommand{\HOM}{\mathop{\mathscr{H}{\kern -3pt {om}}}\nolimits}
\newcommand{\calA}{\mathscr{A}}
\newcommand{\calC}{\mathscr{C}}
\newcommand{\calH}{\mathscr{H}}
\newcommand{\calL}{\mathscr{L}}
\newcommand{\calM}{\mathscr{M}}
\newcommand{\calN}{\mathscr{N}}
\newcommand{\calX}{\mathscr{X}}
\newcommand{\calK}{\mathscr{K}}
\newcommand{\calD}{\mathscr{D}}
\newcommand{\calY}{\mathscr{Y}}
\newcommand{\calF}{\mathscr{F}}
\newcommand{\calE}{\mathscr{E}}
\newcommand{\CN}{\mathcal{N}}
\newcommand{\CCC}{\mathcal{C}}
\newcommand{\CM}{\mathcal{M}}
\newcommand{\CQ}{\mathcal{Q}}

\newcommand{\chari}{\mathop{\mathrm{char}}\nolimits}
\newcommand{\ch}{\mathop{\mathrm{ch}}\nolimits}
\newcommand{\CH}{\mathop{\mathrm{CH}}\nolimits}
\newcommand{\supp}{\mathop{\mathrm{supp}}\nolimits}
\newcommand{\codim}{\mathop{\mathrm{codim}}\nolimits}
\newcommand{\td}{\mathop{\mathrm{td}}\nolimits}
\newcommand{\Span}{\mathop{\mathrm{Span}}\nolimits}
\newcommand{\Gal}{\mathop{\mathrm{Gal}}\nolimits}
\newcommand{\sym}{\mathop{\mathrm{Sym}}\nolimits}
\newcommand{\GL}{\mathop{\mathrm{GL}}\nolimits}
\newcommand{\SL}{\mathop{\mathrm{SL}}\nolimits}
\newcommand{\IM}{\mathop{\mathrm{Im}}\nolimits}
\newcommand{\piet}{{\pi_1^{\rm \acute{e}t}}}
\newcommand{\Het}[1]{{H_{\rm \acute{e}t}^{{#1}}}}
\newcommand{\Hfl}[1]{{H_{\rm fl}^{{#1}}}}
\newcommand{\Hcris}[1]{{H_{\rm cris}^{{#1}}}}
\newcommand{\HdR}[1]{{H_{\rm dR}^{{#1}}}}
\newcommand{\hdR}[1]{{h_{\rm dR}^{{#1}}}}
\newcommand{\loc}{{\rm loc}}
\newcommand{\et}{{\rm \acute{e}t}}
\newcommand{\defin}[1]{{\bf #1}}

\ifthenelse{\equal{1}{1}}{
\ifthenelse{\equal{2}{2}}{
\newcommand{\blue}[1]{{\color{blue}#1}}
\newcommand{\green}[1]{{\color{green}#1}}
\newcommand{\red}[1]{{\color{red}#1}}
\newcommand{\cyan}[1]{{\color{cyan}#1}}
\newcommand{\magenta}[1]{{\color{magenta}#1}}
\newcommand{\yellow}[1]{{\color{yellow}#1}} 
}{
\newcommand{\blue}[1]{#1}
\newcommand{\green}[1]{#1}
\newcommand{\red}[1]{#1}
\newcommand{\cyan}[1]{#1}
\newcommand{\magenta}[1]{#1}
\newcommand{\yellow}[1]{#1} 
}
}{
\newcommand{\blue}[1]{}
\newcommand{\green}[1]{}
\newcommand{\red}[1]{}
\newcommand{\cyan}[1]{}
\newcommand{\magenta}[1]{}
\newcommand{\yellow}[1]{} 
}

\renewcommand{\HH}{{\rm{H}}}

\title{Trigonal Canonically Fibered Surfaces}

\date{May 11, 2025}

\thanks{Xi Chen and Nathan Grieve are each partially supported by their respective Discovery Grants from the Natural Sciences and Engineering Research Council of Canada}

\author{Houari Benammar Ammar}
\address{D\'{e}partement de math\'{e}matiques, Universit\'{e} du Qu\'{e}bec \`a Montr\'{e}al,
Local PK-5151, 201 Avenue du Pr\'{e}sident-Kennedy,
Montr\'{e}al, QC, H2X 3Y7, Canada}
\email{benammar\_ammar.houari@courrier.uqam.ca}

\author{Xi Chen}
\address{Department of Mathematical and Statistical Sciences\\
University of Alberta\\
Edmonton, Alberta T6G 2G1, Canada}
\email{xichen@math.ualberta.ca}

\author{Nathan Grieve}
\address{Department of Mathematics \& Statistics,
Acadia University, Huggins Science Hall, Room 130
12 University Avenue
Wolfville, Nova Scotia, B4P 2R6
Canada; 
School of Mathematics and Statistics, 4302 Herzberg Laboratories, Carleton University, 1125 Colonel By Drive, Ottawa, ON, K1S 5B6, Canada; 
D\'{e}partement de math\'{e}matiques, Universit\'{e} du Qu\'{e}bec \`a Montr\'{e}al, Local PK-5151, 201 Avenue du Pr\'{e}sident-Kennedy, Montr\'{e}al, QC, H2X 3Y7, Canada;
Department of Pure Mathematics, University of Waterloo, 200 University Avenue West, Waterloo, ON, N2L 3G1, Canada}
\email{nathan.m.grieve@gmail.com}


\begin{abstract}
We fix some gaps of a proof of Xiao's conjecture on canonically fibered surfaces of relative genus $5$ by the second author. Our argument simplifies the original proof and gives a much better bound on the geometric genus of the surface. Also we apply the same argument to canonically fibered surfaces of relative genus $3$ and $4$ to obtain some Noether-type inequalities for these surfaces.
\end{abstract}

\maketitle
\tableofcontents

\section{Introduction}

A {\em canonically fibered surface} is a rational map $f: X\dashrightarrow C$ satisfying
\begin{enumerate}
\item[C1.]
$X$ and $C$ are irreducible, smooth and projective over $\BC$ of dimension
$\dim X = 2$ and $\dim C = 1$, respectively,
\item[C2.]
$f$ has connected fibers, $X$ is of general type and
\item[C3.]
$|K_X| = f^* \D + {\mathcal N}$ for a base point free (bpf) linear series $\D$ on $C$ of $\dim \D \ge 1$ and ${\mathcal N}$ the fixed part of $|K_X|$. 
\end{enumerate}

Usually, we resolve $f$ and further assume that

\begin{enumerate}
\item[C4.]
$f$ is regular and
\item[C5.]
$K_X$ is relatively nef over $C$ and we write $K_X$ as the sum of its moving part and fixed part
\begin{equation}\label{TRIGCFSE000}
K_X = f^* D + N
\end{equation}
in $\Pic(X)$, where $D\in \Pic(C)$ is an ample and base point free (bpf) divisor on $C$ of $\deg D = d$ and
$N$ is an effective divisor on $X$.
\end{enumerate}

Using the Miyaoka-Yau inequality, A. Beauville proved that the general fibers of $f$ are curves of genus $\le 5$
when the geometric genus $p_g(X) = h^0(K_X)$ of $X$ is sufficiently large \cite{Beauville:1979}. G. Xiao further conjectured \cite[Problem 6]{problem}:

\begin{Conjecture}[Xiao]\label{CONJXIAO}
Let $f: X\dashrightarrow C$ be a canonically fibered surface satisfying C1-C3.
Then $X_p = f^{-1}(p)$ is a curve of genus $\le 4$ for
a general point $p\in C$ when $p_g(X) \gg 1$.
\end{Conjecture}

Given Beauville's result, to prove Xiao's conjecture, it suffices to
show that canonically fibered surfaces whose general fibers are curves of genus $5$ have bounded geometric genus. So we further assume that

\begin{enumerate}
\item[C6.]
$X_p = f^{-1}(p)$ is a smooth projective
curve of genus $5$ for $p\in C$ general.
\end{enumerate}

A proof of Conjecture \ref{CONJXIAO} was claimed by the second author in \cite{CHEN20171033}. However, there are gaps in the proof. The main issue is that the general fibers of $f$ might be trigonal. If that happens, the canonical embedding of $X_p$ is not an intersection of three quadrics, a hypothesis wrongly assumed in \cite{CHEN20171033}. The purpose of this paper is to correct this mistake by mainly proving Xiao's conjecture when the general fibers of $f$ are trigonal. In addition, we will give a simplified proof of Conjecture \ref{CONJXIAO} when the general fibers of $f$ are nontrigonal. Actually this paper is independent of \cite{CHEN20171033}, although we borrow many techniques from that paper.

Our main theorem is 

\begin{Theorem}\label{TRIGCFSTHMMAIN}
Let $f: X \to C$ be a canonically fibered surface satisfying C1-C6. If the general fibers of $f$ are non-hyperelliptic, then $g(C) = 0$ and $p_g(X) \le 50$.
\end{Theorem}

X.T. Sun proved Xiao's conjecture when the general fibers of $f: X\to C$ are hyperelliptic \cite[Corollary 1]{Sun1994}: If $f: X \to C$ is a canonically fibered surface satisfying C1-C6 and the general fibers of $f$ are hyperelliptic, then $g(C)= 0$ and $p_g(X) \le 52$.
A combination of Sun's and our results settles Xiao's Conjecture \ref{CONJXIAO} with an upper bound $p_g(X) \le 52$ on the geometric genus of a canonically fibered surface of relative genus $5$.

We will also study canonically fibered surfaces $f:X \to C$ of relative genus $3$ and $4$. For these surfaces, we are trying to find a constant $c$ such that
$$
K_X^2 \ge c p_g(X) + O(1)
$$
This is another question asked by G. Xiao in \cite[Problem 6]{problem} (see also \cite{Sun1994}). We obtain

\begin{Theorem}\label{TRIGCFSTHMG36}
Let $f: X\to C$ be a canonically fibered surface satisfying C1-C5 with $g = g(X_t) = 3$ or $4$ for $t\in C$ general.
If $f$ is a non-hyperelliptic fibration, then
\begin{equation}\label{TRIGCFSE408}
K_X^2 \ge \begin{cases}
\dfrac{16}3 d + \dfrac{10}3 \deg K_C & \text{for } g=3\\ \vspace{-6pt}\\
\dfrac{22}3 d + \dfrac{14}3 \deg K_C & \text{for } g=4 \text{ and }
d \ge -\dfrac{7}{2}\deg K_C
\end{cases} 
\end{equation}
\end{Theorem}

\section{Outline of the Proof of Theorem \ref{TRIGCFSTHMMAIN}}

\subsection{Lower bound for $\Gamma^2$}
 
Let us collect all of the known results on canonically fibered surfaces $f:X\to C$ satisfying C1-C6. Most of the following results are due to G. Xiao \cite{Xiao1985}, with some improvements that are known to experts. We are not claiming any originality on these results.

Suppose that $f: X\to C$ is a canonically fibered surface satisfying C1-C6. Then

\begin{itemize}
\item $C$ is either rational or elliptic.
\item We have
\begin{equation}\label{TRIGCFSE002}
\begin{aligned}
f_* \omega_X =& f_* \omega_f \otimes \omega_C = L \oplus M\\
\text{where } & h^0(L) = h^0(K_X),\ M\otimes \omega_C^{-1} \text{ is semi-positive and}\\ & h^1(M) = h^1(\OO_X).
\end{aligned}
\end{equation}
Here $\omega_X = K_X$ and $\omega_C$ are the dualizing sheaves of $X$ and $C$ and $\omega_f$ is the relative dualizing sheaf of $f$.
\item We have $N = 8 \Gamma + V$ in \eqref{TRIGCFSE000}. That is,
\begin{equation}\label{TRIGCFSE001}
K_X = f^* D + 8 \Gamma + V
\end{equation}
where $V$ is effective, $f_* V = 0$ and $\Gamma$ is a section of $f$ satisfying 
\begin{equation}\label{TRIGCFSE530}
\begin{aligned}
\boxed{-\frac{4}{33} d + \frac{67}{33} \deg K_C} &\le \deg K_C - K_X \Gamma = \Gamma^2 = -\frac{d}9 + \frac{\deg K_C}9 -
\frac{\Gamma V}9
\\
&\le -\frac{d}{9} + \frac{\deg K_C}{9}
\end{aligned}
\end{equation}
for $d = \deg D$.
\end{itemize}

Note that \eqref{TRIGCFSE530} is an improvement of \cite[Theorem 2.3, p. 1036]{CHEN20171033}. The improvement comes from the application of the logarithmic Bogomolov--Miyaoka--Yau inequality \cite{LOGMY} to $(X,\Gamma)$ instead of $X$ itself. Here is the argument.

\begin{proof}[Proof of \eqref{TRIGCFSE530}]
By adjunction,
$$
\begin{aligned}
\deg K_C = \deg K_\Gamma &= (K_X + \Gamma) \Gamma \\
&= (f^* D + 9\Gamma + V) \Gamma = d + 9\Gamma^2 + \Gamma V
\end{aligned}
$$
Hence
\begin{equation}\label{TRIGCFSE650}
\Gamma^2 = -\frac{d}9 + \frac{\deg K_C}9 - \frac{\Gamma V}9
\end{equation}
and
\begin{equation}\label{TRIGCFSE653}
K_X \Gamma = \frac{d}9 + \frac{8\deg K_C}9 + \frac{\Gamma V}9
\end{equation}
Applying the logarithmic Bogomolov--Miyaoka--Yau inequality \cite{LOGMY} to $(X,\Gamma)$, we obtain
$$
\begin{aligned}
(K_X + \Gamma)^2 \le 3 e(X\backslash \Gamma) &= 3e(X) - 3e(\Gamma)\\
&= 36\chi(\OO_X) - 3K_X^2 + 3\deg K_C
\end{aligned}
$$
where $e(X \backslash)$ is the topological Euler characteristic of $X \backslash$ and where $e(\Gamma)$ is the topological Euler characteristic of $\Gamma$.
Hence
\begin{equation}\label{TRIGCFSE651}
\begin{aligned}
(K_X + \Gamma)^2 + 3K_X^2 &\le 36\chi(\OO_X) + 3\deg K_C\\
&= 36d + 36 - 15 \deg K_C - 36 h^1(\OO_X)
\end{aligned}
\end{equation}
For the left hand side of \eqref{TRIGCFSE651}, we have
\begin{equation}\label{TRIGCFSE652}
\begin{aligned}
&\quad (K_X + \Gamma)^2 + 3K_X^2 = K_X(4K_X + \Gamma) + (K_X+\Gamma)\Gamma\\
&=K_X (4 f^* D + 33\Gamma + 4V) + \deg K_C\\
&= K_X (4f^*D + 33\Gamma) + 4 K_X V + \deg K_C\\
&\ge K_X (4f^*D + 33\Gamma) + \deg K_C = 32d + 33 K_X \Gamma + \deg K_C\\
&= \dfrac{107}3 d  + \dfrac{91}3\deg K_C + \dfrac{11}3 \Gamma V
\end{aligned}
\end{equation}
where we replace $K_X \Gamma$ by the right hand side of \eqref{TRIGCFSE653} at the last step.

The combination of \eqref{TRIGCFSE651} and \eqref{TRIGCFSE652} yields
\begin{equation}\label{TRIGCFSE654}
\begin{aligned}
\Gamma V &\le \dfrac{d}{11} + \dfrac{108}{11}(1-h^1(\OO_X)) - \dfrac{136}{11} \deg K_C
\\
&\le \dfrac{d}{11} + \dfrac{108}{11}(1-h^1(\OO_C)) - \dfrac{136}{11} \deg K_C \le \dfrac{d}{11} - \dfrac{190}{11} \deg K_C
\end{aligned}
\end{equation}
and \eqref{TRIGCFSE530} follows.
\end{proof}

The lower bound of $\Gamma^2$ in \eqref{TRIGCFSE530} is crucial to the proof of our main theorem.

\subsection{Pseudo-canonical model}\label{TRIGCFSSUBSECPCM}

Let $f: X \to C$ be a canonically fibered surface satisfying C1-C6. We have a rational map
\begin{equation}\label{TRIGCFSE552}
\begin{tikzcd}
\tau: X \arrow[dashed]{r} & W = \PP (f_* \LL)^\vee
= \Proj (S^\bullet f_*\LL) = \Proj\left(\bigoplus \operatorname{Sym}^\bullet f_* \LL\right),
\end{tikzcd} 
\end{equation}
where $\LL = \CO_X(8\Gamma)$ with $\Gamma$ given in \eqref{TRIGCFSE001}. The proper transform $Y = \tau(X)$ of $X$ under $\tau$, is called a {\em pseudo relative canonical model} of $X/C$ in \cite{CHEN20171033}.

The study of $\tau$ and $Y$ form the core of the argument in \cite{CHEN20171033}. However, due to the missing case of trigonal fibrations, the statements and proofs of the structure theorems of $Y$ \cite[Theorem 4.2 and 4.7]{CHEN20171033} are wrong. We correct them in the following theorem:

\begin{Theorem}\label{TRIGCFSTHMPCM}
Let $f: X \to C$ be a canonically fibered surface satisfying C1-C6, $\tau$ be the rational map given by \eqref{TRIGCFSE552} and $Y = \tau(X)$.
If the general fibers of $f$ is not hyperelliptic, then
\begin{itemize}
\item $\tau$ is regular;
\item the general fibers of $Y/C$ are smooth;
\item for every $p\in C$, the fiber $Y_p$ of $Y$ over $p$ is an integral curve of arithmetic genus $p_a(Y_p) = 5$ and degree $8$ in $W_p\cong \PP^4$ and one of the following:
\begin{enumerate}
\item[{\bf A1}.] a complete intersection of three quadrics in $W_p$,
\item[{\bf A2}.] a curve lying on a rational normal scroll, or
\item[{\bf A3}.] a curve lying on a cone over a rational normal curve in $\PP^3$.
\end{enumerate}
\end{itemize}
\end{Theorem}

Recall that a rational normal scroll in $\PP^4$ is a smooth linearly non-degenerate cubic surface isomorphic to $\FF_1$, where 
$\FF_n = \PP (\OO_{\PP^1} \oplus \OO_{\PP^1}(-n))$ are Hirzebruch surfaces. It is the embedding $\FF_1\hookrightarrow \PP^4$ given by $|B + 2F|$, where
$B$ and $F$ are the effective generators of $\Pic(\FF_1)$ satisfying $B^2 = -1$, $BF=1$ and $F^2 = 0$.

We call a cone over a rational normal curve in $\PP^3$ a {\em cubic cone}. If we resolve the vertex of a cubic cone, we obtain $\FF_3$. Later we will prove that an integral linearly non-degenerate cubic surface in $\PP^4$ is either a rational normal scroll or a cubic cone. So the fibers $Y_p$ of type A2 and A3 in the above theorem can be described as curves lying on an integral linearly non-degenerate cubic surface in $\PP^4$.

From now on, we will frequently refer to the types of $Y_p$ in the theorem: we say that $Y_p$ has type A1, A2 or A3 if it is a complete intersection of three quadrics, a curve lying on a rational normal scroll or a curve lying on a cubic cone, respectively.

If the general fibers of $f$ are not trigonal, then 
the general fibers of $Y/C$ are of type A1 and its special fibers are of type A2 or A3. Otherwise, if the general fibers of $f$ are trigonal, in which case we will call $f: X\to C$ a {\em trigonal fibration}, then
by the classical result of B. Saint-Donat
\cite{Saint-Donat1973}, $Y_p$ is contained in a rational normal scroll for $p\in C$ general. So the general fibers of $Y/C$ are of type A2 and its special fibers are of type A3, as we will see later.

We will prove Theorem \ref{TRIGCFSTHMPCM} in the next section. It forms the technical core of this paper. The fact that $Y_p$ is reduced is quite difficult to prove, especially when $f$ is trigonal. Here we will show how to derive our main theorem from Theorem \ref{TRIGCFSTHMPCM}. It is done via a corollary:

\begin{Corollary}\label{TRIGCFSCORPCM}
Let $f: X \to C$ be a canonically fibered surface satisfying C1-C6. 
If the general fiber of $f$ is not hyperelliptic, then
for either $a = 4$ or $5$, $f_*\OO_X(a\Gamma)$ has rank $2$ and
the rational map
\begin{equation}\label{TRIGCFSCORPCME000}
\begin{tikzcd}
\alpha: X \arrow[dashed]{r} & S = \PP (f_* \OO_X(a\Gamma))^\vee
\end{tikzcd} 
\end{equation}
is regular.
\end{Corollary}

\subsection{Corollary \ref{TRIGCFSCORPCM} implies Theorem \ref{TRIGCFSTHMMAIN}}

Let us first show how to derive our main theorem from Corollary \ref{TRIGCFSCORPCM}. We are following the same argument in \cite[Theorem 2]{Sun1994}.

Since $f_*\OO_X(a\Gamma)$ has rank $2$ and $\alpha$ is regular, $\alpha$ is a generically finite morphism of degree $a$ from $X$ to a ruled surface $S$ over $C$. Let
$$\Lambda\in |\OO_S(1)| \cong |\OO_X(a\Gamma)|$$
be the unique section. Then
$$
a\Gamma = \alpha^* \Lambda.
$$
It follows that
\begin{equation}\label{TRIGCFSCORPCME001}
\Gamma^2 = \dfrac{1}a \Lambda^2.
\end{equation}

Clearly, $f_* \OO_X(a\Gamma)$ is a subsheaf of $f_* \OO_X(8\Gamma)$, which is in turn a subsheaf of $f_* K_X\otimes \OO_C(-D)$.
By \eqref{TRIGCFSE002},
\begin{equation}\label{TRIGCFSCORPCME002}
\begin{aligned}
&\quad f_* \OO_X(a\Gamma) \subset f_* \OO_X(8\Gamma) \subset f_* K_X\otimes \OO_C(-D) = \OO_C \oplus L^{-1}\otimes M\\
&\text{with } \mu_{\max}(L^{-1}\otimes M) \le -d + \dfrac{\deg K_C}2
\end{aligned}
\end{equation}
where $\mu_{\max}(\VV)$ is the maximal slope of a vector bundle $\VV$ on $C$ in its Harder-Narasimhan filtration. And since $h^0(\OO_X(a\Gamma)) = 1$, we must have 
\begin{equation}\label{TRIGCFSCORPCME018}
f_* \OO_X(a\Gamma) = \OO_C \oplus \calL
\hspace{12pt}\text{with } \mu_{\max}(\calL) \le -d + \dfrac{\deg K_C}2
\end{equation}
Combining \eqref{TRIGCFSCORPCME001} and \eqref{TRIGCFSCORPCME018}, we obtain
\begin{equation}\label{TRIGCFSCORPCME003}
\Gamma^2 \le  \dfrac{1}a \left(-d + \dfrac{\deg K_C}2\right).
\end{equation}
Then, together with \eqref{TRIGCFSE530}, we have
\begin{equation}\label{TRIGCFSCORPCME004}
-\frac{4}{33} d + \frac{67}{33} \deg K_C \le \Gamma^2 \le  \dfrac{1}a \left(-d + \dfrac{\deg K_C}2\right).
\end{equation}
Therefore,
$$
d \le - \dfrac{134 a - 33}{66-8a} \deg K_C \le -\dfrac{49}2 \deg K_C.
$$
So $g(C) = 0$ and $p_g(X) = d+1 \le 50$. This proves Theorem \ref{TRIGCFSTHMMAIN}.

\subsection{Theorem \ref{TRIGCFSTHMPCM} implies Corollary \ref{TRIGCFSCORPCM}}

Next, let us show how to derive Corollary \ref{TRIGCFSCORPCM} from Theorem \ref{TRIGCFSTHMPCM}.

For all vector bundles $\VV$ on $X$, the natural map
\begin{equation}\label{TRIGCFSTHMPCME000}
\begin{tikzcd}
f_* \VV \Big|_p = f_* \VV \otimes \OO_p \ar[hook]{r} & H^0(X_p, \VV|_p)
\end{tikzcd}
\end{equation}
is an injection. So we regard $f_* \VV \otimes \OO_p$ as a subspace of 
$H^0(X_p, \VV|_p)$. To prove Corollary \ref{TRIGCFSCORPCM}, it suffices to show that
\begin{equation}\label{TRIGCFSTHMPCME001}
f_* \OO_X(a\Gamma) \otimes \OO_p \not\subset H^0(\OO_{X_p}((a-1)\Gamma))
\end{equation}
for every point $p\in C$, where $f_* \OO_X(a\Gamma) \otimes \OO_p$ and $H^0(\OO_{X_p}((a-1)\Gamma))$ are regarded as subspaces $H^0(\OO_{X_p}(a\Gamma))$, or equivalently, as subspaces of $H^0(\OO_{X_p}(8\Gamma))$. So we may put \eqref{TRIGCFSTHMPCME001} in the equivalent form
\begin{equation}\label{TRIGCFSTHMPCME002}
f_* \OO_X(a\Gamma) \otimes \OO_p \not\subset H^0(\OO_{X_p}((a-1)\Gamma))
\cap f_* \OO_X(8\Gamma) \otimes \OO_p
\end{equation}
with all spaces in $H^0(\OO_{X_p}(8\Gamma))$.

Let $\Lambda \in |\OO_W(1)| \cong |\OO_X(8\Gamma)|$ be the unique section. Then
$$
\tau^* \Lambda = 8 \Gamma.
$$
Let $\Gamma_Y = Y \cap \Lambda$. Then $\tau^{-1}(\Gamma_Y) = \Gamma$.

Since $Y_p$ is reduced for all $p$, $Y$ is Cohen-Macaulay. And since $Y_p$ is smooth for $p$ general, $Y$ is smooth in codimension $1$. Then by Serre's criterion, $Y$ is normal. And since $\tau: X\to Y$ is finite over $\Gamma_Y$, it must be a local isomorphism along $\Gamma_Y$, i.e., there exists an open neighborhood $U\subset Y$ of $\Gamma_Y$ such that $\tau: \tau^{-1}(U)\to U$ is an isomorphism. It follows that $Y_p$ is smooth at the point $Y_p\cap \Lambda$ for all $p\in C$.

For every $p\in C$ and $0\le b\le 8$, we may identify
\begin{equation}\label{TRIGCFSTHMPCME003}
\begin{aligned}
&\quad H^0(\OO_{X_p}(b\Gamma))
\cap f_* \OO_X(8\Gamma) \otimes \OO_p
\\
&= \big\{A\in H^0(\OO_{W_p}(1)): (A.Y_p)_q \ge 8-b \big\}
\end{aligned}
\end{equation}
where $q = Y_p\cap \Lambda$. In other words, we can identify the left hand side of \eqref{TRIGCFSTHMPCME003} with the space of hyperplanes
in
$$
H^0(\OO_{W_p}(1)) \cong f_* \OO_X(8\Gamma) \otimes \OO_p
$$
that are tangent to $Y_p$ at $q$ with multiplicity at least $8-b$. 
Using this identification, we can prove
\begin{equation}\label{TRIGCFSTHMPCME004}
\begin{aligned}
& \dim H^0(\OO_{X_p}(4\Gamma))
\cap f_* \OO_X(8\Gamma) \otimes \OO_p = 1\\
&\hspace{60pt}\text{if } Y_p \text{ is of type A2 or A3}
\end{aligned}
\end{equation}
and
\begin{equation}\label{TRIGCFSTHMPCME005}
\dim H^0(\OO_{X_p}(3\Gamma))
\cap f_* \OO_X(8\Gamma) \otimes \OO_p = 1
\hspace{12pt}\text{if } Y_p \text{ is of type A1}.
\end{equation}

Suppose that $J = Y_p$ is of type A2. That is, $J$ lies on a rational normal scroll $R\subset W_p$. Since $J$ has arithmetic genus $5$ and degree $8$, it is easy to see that $J\in |3B + 5F|$ on $R\cong \FF_1$.

If \eqref{TRIGCFSTHMPCME004} fails, then
$$
\dim H^0(\OO_{X_p}(4\Gamma))
\cap f_* \OO_X(8\Gamma) \otimes \OO_p \ge 2
$$
Using \eqref{TRIGCFSTHMPCME003}, we see that there exists a hyperplane $\Lambda'\ne \Lambda_p$ in $W_p$ such that
$$
(\Lambda' . J)_q \ge 4
$$
for $q = J\cap \Lambda$.

Let $\Lambda_R = \Lambda . R$ be the scheme-theoretical intersection between $\Lambda$ and $R$. Then $\Lambda_R \in |B + 2F|$ on $R\cong \FF_1$. We claim that $\Lambda_R$ is integral. Otherwise, $\Lambda_R$ has two distinct irreducible components $G_1$ and $G_2$ with $G_1\in |F|$ and $G_2\in |B|$ or $|B+F|$. Since $q$ is the only intersection between $\Lambda_R$ and $J$ on $R$, we have
$$
q = G_1\cap J = G_2\cap J.
$$
Since $J\in |3B + 5F|$, $G_1.J \ge 2$ and $G_2.J\ge 2$. Hence
$$
(J.G_1)_q \ge 2 \hspace{12pt}\text{and}\hspace{12pt} (J.G_2)_q \ge 2
$$
on $R$. This is impossible since $J$, $G_1$ and $G_2$ are smooth at $q$ and $G_1$ and $G_2$ meet transversely at $q$. Therefore, $\Lambda_R$ is integral.

Let $\Lambda_R' = \Lambda' . R$. Since $R$ is linearly non-degenerate,
$\Lambda_R \ne \Lambda_R'$. And since $\Lambda_R$ is integral, $\Lambda_R$ and $\Lambda_R'$ meet properly on $R$. Since
$$
(\Lambda_R.J)_q = 8 \hspace{12pt}\text{and}\hspace{12pt}
(\Lambda_R'.J)_q \ge 4
$$
and $J$ is smooth at $q$, we must have
$$
(\Lambda_R.\Lambda_R')_q \ge 4.
$$
But $\Lambda_R . \Lambda_R' = \deg R = 3$, which is a contradiction.
This proves \eqref{TRIGCFSTHMPCME004} when $Y_p$ is of type A2.

Suppose that $J = Y_p$ is of type A3. That is, $J$ lies on a cubic cone $R\subset W_p$. Let $\pi: \widehat{R}\cong \FF_3\to R$ be the resolution 
of $R$ at the vertex $o$ and let $\widehat{J} \subset \widehat{R}$ be the proper transform of $J$ under $\pi$. Since $\pi^* \OO_R(1) = \OO_{\widehat{R}} (B+3F)$ and $\deg J = 8$, $\widehat{J} (B+3F) = 8$ on $\widehat{R}$, where $B$ and $F$ are effective generators of $\Pic(\FF_3)$ with $B^2 = -3$, $BF=1$ and $F^2 = 0$.
Therefore, $\widehat{J} \in |B + 8F|$ or $|2B + 8F|$. Hence $p_a(\widehat{J}) = 0$ or $4$. In any event, $p_a(\widehat{J}) < p_a(J) = 5$. So $J$ must be singular at $o$. It follows that $q \ne o$ for $q = J\cap \Lambda$. That is, $R$ is smooth at $q$ and $\Lambda$ does not pass through $o$.

If \eqref{TRIGCFSTHMPCME004} fails, then
there exists a hyperplane $\Lambda'\ne \Lambda_p$ in $W_p$ such that
$$
(\Lambda' . J)_q \ge 4.
$$
We choose $\Lambda'$ to be a general member of the space
$$
\big\{A\in H^0(\OO_{W_p}(1)): (A.J)_q \ge 4 \big\}.
$$
Since $\Lambda$ does not pass through $o$, neither does $\Lambda'$.
So both $\Lambda$ and $\Lambda'$ meet $R$ transversely along rational normal curves in $\Lambda_p\cong \PP^3$ and $\Lambda'\cong\PP^3$, respectively. And since $R$ is linearly non-degenerate, $\Lambda_R = \Lambda . R$ and $\Lambda_R' = \Lambda' . R$ meet properly on $R$.
Again, since
$$
(\Lambda_R.J)_q = 8 \hspace{12pt}\text{and}\hspace{12pt}
(\Lambda_R'.J)_q \ge 4
$$
and $J$ is smooth at $q$, we must have
$$
(\Lambda_R.\Lambda_R')_q \ge 4.
$$
But $\Lambda_R . \Lambda_R' = \deg R = 3$, which is a contradiction.
This proves \eqref{TRIGCFSTHMPCME004} when $Y_p$ is of type A3.

Suppose that $J = Y_p$ is of type A1. That is, $J$ is cut out by three quadrics.

If \eqref{TRIGCFSTHMPCME005} fails, then
there exists a hyperplane $\Lambda'\ne \Lambda_p$ in $W_p$ such that
$$
(\Lambda' . J)_q \ge 5
$$
for $q=J\cap \Lambda$.
We choose $\Lambda'$ to be a general member of the space
$$
\big\{A\in H^0(\OO_{W_p}(1)): (A.J)_q \ge 5 \big\}.
$$

Let $\CQ \subset |\OO_{W_p}(2)|$ be the linear system of quadrics in $W_p$ containing $J$. Since $J = \operatorname{Bs}(\CQ)$ and $q = J\cap \Lambda\cap \Lambda'$,
$$
\dim (Q_1\cap Q_2\cap \Lambda\cap \Lambda') = 0
$$
for two general members $Q_1$ and $Q_2$ of $\CQ$. That is, $Q_1, Q_2, \Lambda$ and $\Lambda'$ meet properly in $W_p$.

Let $S = Q_1\cap Q_2$. Since $J$ is smooth at $q$, $S$ is smooth at $q$. 
Again, since
$$
(\Lambda.J)_q = 8 \hspace{12pt}\text{and}\hspace{12pt}
(\Lambda'.J)_q \ge 5
$$
and $J$ is smooth at $q$, we must have
$$
(S.\Lambda.\Lambda')_q \ge 5.
$$
But $S, \Lambda$ and $\Lambda'$ meet properly in $W_p$ and
$S.\Lambda.\Lambda' = 4$, which is a contradiction.
This proves \eqref{TRIGCFSTHMPCME005}.

When $Y_p$ is of type A1,
$\OO_{Y_p}(8q) = \omega_{Y_p} = \OO_W(1)\otimes \OO_{Y_p}$ and hence
$$
H^0(\OO_{Y_p}(8q)) = H^0(\OO_{W_p}(1)) = f_* \OO_X(8\Gamma)\otimes \OO_p.
$$
So \eqref{TRIGCFSTHMPCME005} is equivalent to
\begin{equation}\label{TRIGCFSTHMPCME006}
h^0(\OO_{Y_p}(3q)) = 1
\hspace{12pt}\text{if } Y_p \text{ is of type A1}.
\end{equation}

From \eqref{TRIGCFSTHMPCME004} and \eqref{TRIGCFSTHMPCME005}, $f_* \OO_X(4\Gamma)$ has rank either $2$ or $1$. Using this fact, we are going to deduce that $\alpha$ is regular for $a = 4$ or $a = 5$. To achieve this, we consider separately the two cases that $f_* \OO_X(4\Gamma)$ has rank $1$ respectively rank $2$. In fact when $f_* \OO_X(4\Gamma)$ has rank $2$, then $\alpha$ is regular when $a = 4$. When $f_* \OO_X(4\Gamma)$ has rank $1$, $\alpha$ is regular when $a = 5$.

If $f_* \OO_X(4\Gamma)$ has rank $2$, then \eqref{TRIGCFSTHMPCME004}
and \eqref{TRIGCFSTHMPCME005} imply that
$$
f_* \OO_X(4\Gamma) \otimes \OO_p \not\subset H^0(\OO_{X_p}(3\Gamma))
\cap f_* \OO_X(8\Gamma) \otimes \OO_p
$$
since
$$
\dim f_* \OO_X(4\Gamma) \otimes \OO_p = 2 > 1= \dim H^0(\OO_{X_p}(3\Gamma))
\cap f_* \OO_X(8\Gamma) \otimes \OO_p
$$
for all $p\in C$. Thus, the rational map
$$
\begin{tikzcd}
\alpha: X \arrow[dashed]{r} & S = \PP (f_* \OO_X(4\Gamma))^\vee
\end{tikzcd} 
$$
is regular.

Suppose that $f_* \OO_X(4\Gamma)$ has rank $1$. Then $f_* \OO_X(5\Gamma)$ has rank $2$ since
$$
\dim f_* \OO_X(8\Gamma) \otimes \OO_p - 3\le
\dim f_* \OO_X(5\Gamma) \otimes \OO_p\le
\dim f_* \OO_X(4\Gamma) \otimes \OO_p + 1
$$
for all $p\in C$.

By \eqref{TRIGCFSTHMPCME004},
$$
f_* \OO_X(5\Gamma) \otimes \OO_p \not\subset H^0(\OO_{X_p}(4\Gamma))
\cap f_* \OO_X(8\Gamma) \otimes \OO_p
$$
for $Y_p$ of type A2 or A3. So
the rational map
$$
\begin{tikzcd}
\alpha: X \arrow[dashed]{r} & S = \PP (f_* \OO_X(5\Gamma))^\vee
\end{tikzcd} 
$$
is regular along $X_p$ if $Y_p$ is of type A2 or A3.

Suppose that $J = Y_p$ is of type A1. That is, $J$ is a complete intersection of three quadrics in $W_p \cong \PP^4$. Then
$$
\OO_J(4q)\otimes \OO_J(4q) = \OO_J(8q) = \omega_J
$$
for $q = J\cap \Lambda$.
By a theorem of J. Harris on theta characteristics \cite{38d847b1-588f-3bb6-91c3-8f7be692ac65}, we have
$$
h^0(\OO_J(4q)) \equiv \dim f_* \OO_X(4\Gamma) \otimes \OO_p \equiv 1 \ (\text{mod } 2).
$$
On the other hand, $h^0(\OO_J(4q)) \le 2$ by \eqref{TRIGCFSTHMPCME006}. So $h^0(\OO_J(4q)) = 1$ and it follows that
$$
f_* \OO_X(5\Gamma) \otimes \OO_p \not\subset H^0(\OO_{X_p}(4\Gamma))
\cap f_* \OO_X(8\Gamma) \otimes \OO_p.
$$
Consequently, $\alpha$ is regular along $X_p$ if $Y_p$ is of type A1. Thus, $\alpha$ is regular everywhere.

This proves Corollary \ref{TRIGCFSCORPCM}. We are going to prove Theorem \ref{TRIGCFSTHMPCM} in the next section, which will complete the proof of
Theorem \ref{TRIGCFSTHMMAIN}.

\section{Proof of Theorem \ref{TRIGCFSTHMPCM}}

\subsection{Regularity of $\tau$}

Since the general fibers of $f$ are not hyperelliptic, 
the canonical embedding $\tau: X_p\to Y_p$ is an isomorphism
and hence $Y_p$ is smooth for $p\in C$ general.
Also it is clear that $\tau$ is regular outside $\Gamma$. To see that $\tau$ is regular away from $\Gamma$, we follow the same argument in \cite{CHEN20171033} by basing our analysis of $Y_p$, for each $p \in C$. 

To this end, in what follows, for each $p \in C$, we let $B_p$ be the unique component of $X_p$ that meets the section $\Gamma$. Further, $J_p = \tau(B_p)$ denotes, for each $p \in C$, the proper transform of $B_p$ under $\tau$.

In our study of $Y_p$, for each $p \in C$, the following three facts are significant.
\begin{itemize}
\item There is a net of quadrics in $W_p$ containing $Y_p$, for each $p \in C$.
\item For each $p\in C$, $J_p$ is linearly non-degenerate in $W_p\cong \PP^4$.
\item For each $p \in C$, either $\tau$ maps $B_p$ $2$-to-$1$ onto $J_p$ and $J_p$ is a rational normal curve in $W_p\cong \PP^4$ or $\tau$ maps $B_p$ birationally onto $J_p$ and $J_p$ has arithmetic genus $p_a(J_p)\ge 5$.
\end{itemize}
Indeed, since $\deg Y_p = 8$, the regularity of $\tau$ is then readily deduced.

For instance, if $\tau$ maps $B_p$, for some $p \in C$, $2$-to-$1$ onto $J_p$, then $\deg J_p = 4$ and $J_p$ is a rational normal curve. In this case, since $Y_p$ has no other irreducible component aside from $J_p$, it follows that $\tau$ must be regular along $X_p$.

On the other hand, if $\tau$ maps $B_p$, for some $p \in C$, birationally onto $J_p$, then the regularity of $\tau$ along $X_p$ follows upon establishing that $\deg J_p = 8$. 

It turns out that in proving the the regularity of $\tau$, we also will establish that 
\begin{itemize}
\item $Y_p$ has pure dimension $1$ for all $p\in C$ and hence $Y$ is Cohen-Macaulay;
\item $Y_p$ has type A1, A2 or A3 for each $p\in C$;
\item $Y_p$ is an integral curve if $Y_p$ has type A1 or A2;
\item $Y_p$ is an integral curve or a closed subscheme with multiplicity $2$ along a rational normal curve on a cubic cone if $Y_p$ has type A3.
\end{itemize}

Let us now begin our argument to show that $\tau$ is regular. Let $p \in C$. The net of quadrics containing $Y_p$ comes from the exact sequence
\begin{equation}\label{TRIGCFSTHMPCME518}
\begin{tikzcd}
0 \arrow{r} & \CQ \arrow{r} & \operatorname{Sym}^2 f_*\LL \arrow{r} & f_* (\LL^{\otimes 2})
\end{tikzcd}
\end{equation}
for $\LL = \OO_X(8\Gamma)$. Since the general fibers of $f$ are not hyperelliptic, the map $\operatorname{Sym}^2 f_*\LL \to f_* (\LL^{\otimes 2})$ is generically surjective. So $\CQ$ is a vector bundle of rank $3$ over $C$. For each $p\in C$, $\CQ_p$ can be identified with a subspace
$$
\CQ_p\subset H^0(I_{Y_p}(2))
$$
of dimension $3$, where $I_{Y_p}$ is the ideal sheaf of $Y_p$ in $W_p$.
So $|\CQ_p|$ is the net of quadrics containing $Y_p$.

If $X_p$ is non-trigonal, $Y_p$ is the (scheme-theoretic) base locus $\operatorname{Bs}(\CQ_p)$ of $\CQ_p$. That is, $Y_p$ is a complete intersection of three quadrics. However, if $X_p$ is trigonal, $\operatorname{Bs}(\CQ_p)$ is a rational normal scroll in $W_p$ containing $Y_p$. Another possibility, as we will see, is that $\operatorname{Bs}(\CQ_p)$ is a cubic cone.

By the injectivity of the map (see \ref{TRIGCFSTHMPCME000})
$$
\begin{tikzcd}
f_* \LL \Big|_p = f_* \LL \otimes \OO_p \ar[hook]{r} & H^0(X_p, \LL|_p)
\end{tikzcd}
$$
we see that $J_p = \tau(B_p)$ is linearly non-degenerate in $W_p\cong \PP^4$. So $\deg J_p \ge 4$. And since $\deg Y_p = 8$, we see that $\tau$ maps $B_p$ either birationally or $2$-to-$1$ onto $J_p$.

If $\tau$ maps $B_p$ $2$-to-$1$ onto $J_p$, then $\deg J_p = 4$ and $J_p$ is a rational normal curve. In this case, $Y_p$ has no other irreducible component other than $J_p$ since $\deg Y_p = 8$. It follows that $\tau$ is regular along $X_p$. So to prove the regularity of $\tau$, it suffices to prove that $\deg J_p = 8$ when $\tau$ maps $B_p$ birationally onto $J_p$.

When $\tau$ maps $B_p$ birationally onto its proper transform $J_p$ 
we claim that $p_a(J_p) \ge 5$ \cite[(4.13), p. 1041]{CHEN20171033}. From this, we will deduce, using Lemmas \ref{TRIGCFSLEM3QUADRICS} and \ref{TRIGCFSLEMCUBIC} that $J_p = Y_p$. To see that $p_a(J_p) \ge 5$, this follows from a useful lemma \cite[Lemma 4.6, p. 1043]{CHEN20171033}. We will prove a slightly improved version as follows:

\begin{Lemma}\label{TRIGCFSLEMGENUS}
Let $X$ and $Y$ be two surfaces proper and flat over the unit disk
$\Delta = \{|t| < 1\}$ with fibres $X_t$ and $Y_t$ smooth and irreducible
for $t\ne 0$. Denote by $X_0$ the fibre of $X$ over $t=0$.
Let $\tau: X\dashrightarrow Y$ be a birational map with the diagram
\begin{equation}\label{TRIGCFSE561}
\begin{tikzcd}
X \arrow[dashed]{r}{\tau}\arrow{d} & Y\arrow{dl}\\
\Delta
\end{tikzcd}
\end{equation}
such that
\begin{itemize}
\item
$\tau$ is regular on $X\backslash \Sigma$ for a finite set $\Sigma$ of points on $X_0$,
\item
$X_0$ is smooth at each point $x\in \Sigma$, and
\item
for each component $\Gamma$ of $X_0$ satisfying $\Gamma\cap \Sigma\ne \emptyset$, 
$\tau_*(\Gamma\backslash \Sigma) \ne 0$.
\end{itemize}
Let $J$ be the scheme-theoretic proper transform of $X_0$. That is, $J$ is the closure of the scheme-theoretic image of $X_0\backslash \Sigma$ under $\tau$.
Then $p_a(J) \ge g(Y_t)$ for $t\ne 0$.
\end{Lemma}

\begin{proof}
We may replace $X$ and $Y$ by their normalizations $\widetilde{X}$ and $\widetilde{Y}$ with the diagram
$$
\begin{tikzcd}
\widetilde{X} \ar{d}{\nu_{X}}\ar[dashed]{r}{\widetilde{\tau}} & \widetilde{Y} \ar{d}{\nu_{Y}}\\
X \ar[dashed]{r}{\tau} & Y
\end{tikzcd}
$$
Clearly, $\nu_{_X}$ is a local isomorphism at $\Sigma$ and $\widetilde{\tau}$ is regular on $\widetilde{X} \backslash \nu_{X}^{-1}(\Sigma)$. Let $\widetilde{J}$ be the proper transform of $\widetilde{X}_0$ under $\widetilde{\tau}$. We claim that
\begin{equation}\label{TRIGCFSLEMGENUSE000}
p_a(J) \ge p_a(\widetilde{J}) 
\end{equation}

Let $\bar{J}\subset Y$ be the scheme-theoretic image of $\widetilde{J}$ under $\nu_{Y}$. By our definition of $J$, $J$ is a closed subscheme of $\bar{J}$ and $J$ and $\bar{J}$ agree outside of a finite set of points. But $\bar{J}$ might have embedded points not contained in $J$. So
\begin{equation}\label{TRIGCFSLEMGENUSE005}
p_a(J) \ge p_a(\bar{J}) 
\end{equation} 
We claim that
\begin{equation}\label{TRIGCFSLEMGENUSE006}
p_a(\bar{J}) \ge p_a(\widetilde{J}) 
\end{equation}
which obviously implies \eqref{TRIGCFSLEMGENUSE000}.

Fixing an ample divisor $A$ on $Y$, the corresponding Hilbert polynomials of $\bar{J}$ and $\widetilde{J}$ are
\begin{equation}\label{TRIGCFSLEMGENUSE001}
h^0(\OO_{\bar{J}}(nA))= n A.\bar{J} - p_a(\bar{J}) + 1
\end{equation}
and
\begin{equation}\label{TRIGCFSLEMGENUSE002}
h^0(\OO_{\widetilde{J}}(n \nu_{Y}^* A)) = n \nu_{Y}^* A.\widetilde{J} - p_a(\widetilde{J}) + 1
\end{equation}
for $n\gg0$. Since
$$
\begin{aligned}
\nu_{Y}^* H^0(\OO_Y(nA)) &\subset H^0(\OO_{\widetilde{Y}}(n \nu_{Y}^* A)) \hspace{12pt}\text{and}\\
\nu_{Y}^* H^0(\OO_Y(nA))\cap H^0 (\OO_{\widetilde{Y}}(n\nu_{Y}^* A)\otimes I_{\widetilde{J}})
&= \nu_{Y}^* H^0(\OO_Y(nA) \otimes I_{\bar{J}}),
\end{aligned}
$$
$\nu_Y^*$ induces an injection
$$
\begin{tikzcd}
H^0(\OO_{\bar{J}}(nA))  \ar[hook]{r}{\nu_Y^*} \ar[equal]{d}  & H^0(\OO_{\widetilde{J}}(n \nu_{Y}^* A)) \ar[equal]{d}\\
\dfrac{H^0(\OO_Y(nA))}{H^0(\OO_Y(nA) \otimes I_{\bar{J}})} \ar[hook]{r}{\nu_Y^*} & \dfrac{H^0(\OO_{\widetilde{Y}}(n \nu_{Y}^* A))}{H^0 (\OO_{\widetilde{Y}}(n\nu_{Y}^* A)\otimes I_{\widetilde{J}})}
\end{tikzcd}
$$
where $I_{\bar{J}}$ and $I_{\widetilde{J}}$ are the ideal sheaves of $\bar{J}$ and $\widetilde{J}$ in $Y$ and $\widetilde{Y}$, respectively. Thus,
\begin{equation}\label{TRIGCFSLEMGENUSE003}
h^0(\OO_{\bar{J}}(nA)) \le h^0(\OO_{\widetilde{J}}(n \nu_{Y}^* A)).
\end{equation}
Combining \eqref{TRIGCFSLEMGENUSE001}-\eqref{TRIGCFSLEMGENUSE003} and the fact that
\begin{equation}\label{TRIGCFSLEMGENUSE004}
A.\bar{J} = \nu_{Y}^* A.\widetilde{J},
\end{equation}
we conclude \eqref{TRIGCFSLEMGENUSE006} and thus \eqref{TRIGCFSLEMGENUSE000}. 
We just have to prove $p_a(\widetilde{J}) \ge p_a(\widetilde{Y}_t)$ since it implies that
$$
p_a(J) \ge p_a(\widetilde{J})\ge p_a(\widetilde{Y}_t) = p_a(Y_t).
$$
So we may simply assume that both $X$ and $Y$ are normal.
In addition, we may normalize $Y$ after a base change. After a suitable base change, the normalization of $Y$ has reduced central fiber. So we may further assume that $Y_0$ is reduced.

Let $\tau_0: X_0\dashrightarrow Y$ be the restriction of $\tau$ to $X_0$. Since $X_0$ is smooth at $\Sigma$, $\tau_0$ extends to a regular morphism on $X_0$. By our hypothesis, $\tau_0: X_0\to J$ is finite at $\Sigma$. 

Let
$$
\begin{tikzcd}
\widehat{X}\ar{d}{f} \ar{dr}{\widehat{\tau}}\\
X \ar[dashed]{r}{\tau} & Y
\end{tikzcd}
$$
be a resolution of the rational map $\tau$, where $f$ is a birational morphism consisting of a sequence of blowups of points over $\Sigma$. 

We see that $\tau_0^{-1}(q) = p$ for every $p\in \Sigma$ and $q = \tau_0(p)$. Otherwise, if there exists $p' \ne p$ such that $\tau_0(p') = q$, then the preimage $\widehat{\tau}^{-1}(q)$ has at least two connected components since $\tau_0: X_0\to J$ is finite at $p$. This is impossible since $Y$ is normal and $\widehat{\tau}$ has connected fibers. Therefore, the map $\tau_0: \Sigma\to \tau_0(\Sigma)$ is one-to-one. This also implies that $J$ has a locally irreducible singularity at $\tau_0(p)$ for every $p\in \Sigma$.

For a point $p\in \Sigma$, we let $E_p = f^{-1}(p)$ and $F_p = \widehat{\tau}(E_p)$.
Again, since $\widehat{\tau}$ has connected fibers, $F_p \cap F_q = \emptyset$ for all $p\ne q\in \Sigma$ and $F_p\cap J = \tau_0(p)$ for all $p\in \Sigma$. So
$$
E = \widehat{\tau}^{-1}(F) \hspace{12pt} \text{for } E =\bigcup_{p\in \Sigma} E_p
\text{ and } F =\bigcup_{p\in \Sigma} F_p.
$$

After we apply stable reductions to $X$ and $Y$ along $E$ and $F$, we may assume that $X_0$ and $Y_0$ have simple normal crossings in open neighborhoods of $E$ and $F$. More precisely, we have a diagram
$$
\begin{tikzcd}
\widehat{X}\ar{d}{f} \ar{r}{\widehat{\tau}} & \widehat{Y}\ar{d}{g}\\
X \ar[dashed]{r}{\tau} & Y
\end{tikzcd}
$$
after a base change, where $f: \widehat{X}\to X$ and $g:\widehat{Y}\to Y$ are birational morphisms such that $\widehat{X}_0$ and $\widehat{Y}_0$ have simple normal crossings in open neighborhoods of $E_p = f^{-1}(p)$ and $F_p = \widehat{\tau}(E_p)$ for all $p\in\Sigma$, respectively, and $f: \widehat{X}\backslash E\to X\backslash \Sigma$ and $g: \widehat{Y}\backslash F\to Y\backslash g(F)$ are finite. 

The central fiber $\widehat{Y}_0$ of $\widehat{Y}/\Delta$ is a union
$$
\widehat{Y}_0 = \widehat{J} \cup F
$$
where $\widehat{J}$ is the proper transform of $J$ under $g$. 
The map $g: \widehat{J}\to J$ resolves the singularities of $J$ at $\tau_0(\Sigma)$.
So $p_a(\widehat{J}) \le p_a(J)$.

Since $X_0$ is smooth at every point $p\in \Sigma$, $p_a(E_p) = 0$. Hence $p_a (F_p) = 0$
for all $p\in \Sigma$. Since $\widehat{Y}_0$ has simple normal crossings in an open neighborhood of $F_p$, each $F_p$ meets $\widehat{J}$ transversely at a single point. In addition, $F_p\cap F_q = \emptyset$ for all $p\ne q\in \Sigma$. Therefore,
$$
p_a(J) \ge p_a(\widehat{J}) = p_a(\widehat{Y}_0) = p_a(Y_t).
$$
\end{proof}

Compared with \cite[Lemma 4.6, p. 1043]{CHEN20171033}, we do not require $J$ to be reduced.

Regardless of whether $\tau$ maps $B_p$ birationally onto $J_p$, since $J_p$ is linearly non-degenerate, every quadric of $|\CQ_p|$ is integral, i.e., has rank $\ge 3$. Thus, there do not exist two quadrics $P\ne Q$ in $|\CQ_p|$ of rank $3$ such that $P$ and $Q$ have the same singular locus;
otherwise, there is a quadric in $\Span\{P, Q\}$ of rank $< 3$. 

\begin{Lemma}\label{TRIGCFSLEM3QUADRICS}
Let $\CQ$ be a net of quadrics in $\PP^4$. Suppose that $\operatorname{Bs}(\CQ)$ has dimension one and it contains a linearly non-degenerate integral curve $J$. Then $\operatorname{Bs}(\CQ)$ is reduced at a general point of $J$. In addition, if $p_a(J)\ge 5$, $J = \operatorname{Bs}(\CQ)$.
\end{Lemma}

\begin{proof}
Let $Y = \operatorname{Bs}(\CQ)$. For every point $p\in Y$, a general member of $\CQ$ is smooth at $p$; otherwise, $Y$ is a complete intersection of three cones over quadric surfaces with common vertex $p$, which is a union of lines and hence cannot containing a linearly non-degenerate component. In other words, if we let
$$
D_p = \big\{Q\in \CQ: Q \text{ is singular at } p\}
$$
then $\dim D_p \le 1$ for all $p\in Y$.

If $Y$ contains an irreducible component $\Gamma$ with multiplicity $\ge 4$, then $\Gamma \ne J$ since $\deg J \ge 4$ and $\deg Y = 8$. So $\deg J \le \deg Y - 4\deg \Gamma \le 4$. Hence $\deg J = 4$ and $Y$ is reduced at a general point of $J$.

Suppose that every irreducible component of $Y$ has multiplicity $\le 3$. Then at a general point $p\in Y$, two general quadrics of $\CQ$ intersect transversely at $p$. In other words, $\dim D_p = 0$ for a general point $p\in Y$. It follows that for a general pencil (i.e. lines) $\calL\subset \CQ$, the base locus $\operatorname{Bs}(\calL)$ of $\calL$ is a surface with isolated singularities.

Let $\Sigma \subset Y$ be the set of points $p$ such that $\dim D_p = 1$. Then $\dim \Sigma = 0$. For every $p\in \Sigma$, a general member $Q$ of $D_p$ is a quadric of rank $4$, i.e., a cone over smooth quadric surface. Hence for every pencil $\calL\subset \CQ$ such that $Q\in \calL$ and $\calL\not\subset D_p$, the surface $\operatorname{Bs}(\calL)$ has a
singularity at $p$ whose tangent cone is a conic curve of rank $\ge 2$. For a general pencil $\calL\subset \CQ$, since $\operatorname{Bs}(\calL)$ has only isolated singularities, it has at worst an $A_m$ singularity at $p$.

For every $p\in Y\backslash \Sigma$, $\dim D_p = 0$. For a pencil $\calL \subset \CQ$, if $\calL \cap D_p = \emptyset$, then $\operatorname{Bs}(\calL)$ is smooth at $p$. On the other hand, for a general pencil $\calL\subset \CQ$ such that $\calL \cap D_p \ne \emptyset$, $\operatorname{Bs}(\calL)$ has a singularity at $p$ whose tangent cone is a conic curve of rank $\ge 2$. In other words, there are at most finitely many pencils $\calL\subset \CQ$ such that $\operatorname{Bs}(\calL)$ has a singularity at $p$ whose tangent cone is a conic curve of rank $1$.
Again, for a general pencil $\calL\subset \CQ$, since $\operatorname{Bs}(\calL)$ has only isolated singularities, it has at worst an $A_m$ singularity at every $p\in Y\backslash \Sigma$.

In conclusion, for a general pencil $\calL \subset \CQ$, $S = \operatorname{Bs}(\calL)$ is a surface with at worst $A_m$ singularities. That is, $S$ is a canonical Del Pezzo surface.

If $Y$ is nonreduced along $J$, then $Y = 2J\in |-2K_S|$ on $S$ since $\deg J \ge 4$ and $\deg Y = 8$. Then $J\in |-K_S|$ since $H^1(\OO_S) = 0$ and $K_S$ is Cartier. That is, $J\in |\OO_S(1)|$ is cut out by a hyperplane on $S$. This contradicts the hypothesis that $J$ is linearly non-degenerate. So $Y$ is reduced at a general point of $J$.

Suppose that $p_a(J)\ge 5$. We have $J\subset Y \in |-2K_S|$. It remains to show that $J = Y$.
Suppose that
$$
Y = J + P
$$
Note that both $J$ and $P$ are effective $\QQ$-Cartier divisors. So we cannot apply adjunction directly to $J$ to compute its arithmetic genus.

Let $\pi: \widehat{S}\to S$ be the minimal resolution of $S$. Then
$K_{\widehat{S}} = \pi^* K_S$. Suppose that
$$
\pi^* Y = \pi^* J + \pi^* P = \widehat{J} + R + \pi^* P
$$
where $\widehat{J}$ is the proper transform of $J$ under $\pi$ and $R = \pi^* J - \widehat{J}$ is a $\QQ$-effective divisor supported on the exceptional divisors of $\pi$. Then
\begin{equation}\label{TRIGCFSTHMPCME007}
- n K_{\widehat{S}} - \widehat{J} = (-(n-2) K_{\widehat{S}} + \pi^* P) + R.
\end{equation}

When $n\gg0$,
$$-(n-2) K_{\widehat{S}} + \pi^* P = -(n-2)\pi^* K_S + \pi^* P$$
is nef. So
\eqref{TRIGCFSTHMPCME007}
is the Zariski decomposition of the left hand side since
$$
(-(n-2)\pi^* K_S + \pi^* P) E = 0.
$$
for all exceptional divisors $E$ of $\pi$.

The Hilbert polynomial of $J$ with respect to $-K_S$ is given by
\begin{equation}\label{TRIGCFSTHMPCME008}
\begin{aligned}
\chi(\OO_J(-nK_S)) &= h^0(\OO_S(-nK_S)) - h^0(I_J(-nK_S))
\\
&= h^0(\OO_{\widehat{S}}(-nK_{\widehat{S}})  - h^0(\OO_{\widehat{S}}(-nK_{\widehat{S}} - \widehat{J})).
\end{aligned}
\end{equation}
for $n\gg0$, where $I_J$ is the ideal sheaf of $J$ in $S$.
By \eqref{TRIGCFSTHMPCME007},
\begin{equation}\label{TRIGCFSTHMPCME009}
H^0(\OO_{\widehat{S}}(-nK_{\widehat{S}} - \widehat{J}))
=
H^0(\OO_{\widehat{S}}(-nK_{\widehat{S}} - \widehat{J} - \lfloor R\rfloor))
\end{equation}
By Kwamata-Viehweg vanishing (cf. \cite{E-V}), we have
\begin{equation}\label{TRIGCFSTHMPCME010}
\begin{aligned}
&\quad H^i(\OO_{\widehat{S}}(-nK_{\widehat{S}} - \widehat{J} - \lfloor R\rfloor))\\
&= H^i(\OO_{\widehat{S}}(K_{\widehat{S}} + (-n+1) K_{\widehat{S}} + \pi^*P + \Delta)) = 0
\end{aligned}
\end{equation}
for $i\ge 1$ and $n\gg0$, where $\Delta = R - \lfloor R\rfloor$. Combining \eqref{TRIGCFSTHMPCME008}-\eqref{TRIGCFSTHMPCME010}, we conclude
$$
\chi(\OO_J(-nK_S)) 
= \chi(\OO_{\widehat{S}}(-nK_{\widehat{S}})  - \chi(\OO_{\widehat{S}}(-nK_{\widehat{S}} - \widehat{J} - \lfloor R \rfloor))
$$
and hence
$$
2p_a(J) - 2 = (K_{\widehat{S}} + \widehat{J} + \lfloor R \rfloor)
(\widehat{J} + \lfloor R \rfloor).
$$
Thus,
$$
\begin{aligned}
8\le 2p_a(J) - 2 &= (K_{\widehat{S}} + \widehat{J} + \lfloor R \rfloor)
(\widehat{J} + \lfloor R \rfloor)\\
&= (K_{\widehat{S}} + \pi^* J - \Delta)
(\pi^* J - \Delta)\\
&\le (K_{\widehat{S}} + (\pi^* J - \Delta) + (\pi^* P + \Delta))
(\pi^* J - \Delta)\\
&= -K_{\widehat{S}} (\pi^* J - \Delta)\\
&\le -K_{\widehat{S}} ((\pi^* J - \Delta) + (\pi^* P + \Delta))\\
&= 2K_{\widehat{S}}^2 = 8
\end{aligned}
$$
where the equality holds if and only if
$$
\begin{aligned}
(\pi^* J - \Delta) (\pi^* P + \Delta) = JP - \Delta^2 &= 0\hspace{12pt} \text{and}\\
-K_{\widehat{S}} (\pi^* P + \Delta) &= 0
\end{aligned}
$$
This holds if and only if $P=0$. So $J = Y$.
\end{proof}

The above lemma shows that $J_p = Y_p = \operatorname{Bs}(\CQ_p)$
if $\dim \operatorname{Bs}(\CQ_p) = 1$. In other words, $Y_p$ has type A1 if $\dim \operatorname{Bs}(\CQ_p) = 1$ and $\tau$ is regular along such $Y_p$.

Since every quadric in $\CQ_p$ is integral, $\dim \operatorname{Bs}(\CQ_p) \le 2$. Suppose that $\dim \operatorname{Bs}(\CQ_p) = 2$. 

Let $P\ne Q$ be two distinct quadrics in $\CQ_p$. Then $P\cap Q$ is a surface of degree $4$. Clearly, $P\cap Q$ cannot be an integral surface of degree $4$; otherwise, $\dim \operatorname{Bs}(\CQ_p) = 1$. And since every surface of degree $\le 2$ in $\PP^4$ is linearly degenerate and $J_p\subset P\cap Q$ is linearly non-degenerate, $P\cap Q$ must have a component which is a linearly non-degenerate integral cubic surface. We claim that every linearly non-degenerate integral surface in $\PP^4$ is either a rational normal scroll or a cubic cone, as defined in Section \ref{TRIGCFSSUBSECPCM}.

\begin{Lemma}\label{TRIGCFSLEMCUBIC}
Let $R\subset \PP^4$ be an integral (i.e. reduced and irreducible) surface of degree $3$. If $R$ is not linearly degenerate, i.e., not contained in any hyperplane, then $R$ is either a rational normal scroll or a cubic cone with Hilbert polynomial $(3n+2)(n+1)/2$.
\end{Lemma}

\begin{proof}
Since $R$ is linearly non-degenerate, $R\cap \Lambda$ must be linearly 
non-degenerate and hence a rational normal curve in $\Lambda \cong \PP^3$ for a general hyperplane $\Lambda\subset \PP^4$. This implies that $R$ is non-singular in codimension $1$.

For an arbitrary hyperplane $\Lambda\subset \PP^4$, we see from the exact sequence
$$
\begin{tikzcd}
0 \ar{r} & \OO_R(-1) \ar{r} & \OO_R \ar{r} & \OO_{R\cap \Lambda} \ar{r} & 0
\end{tikzcd}
$$
that the Hilbert polynomial of $R\cap \Lambda$ is
$$
\chi(\OO_{R\cap \Lambda}(n)) = \chi(\OO_R(n)) - \chi(\OO_R(n-1))
$$
which is independent of $\Lambda$. So we have
\begin{equation}\label{TRIGCFSE017}
\chi(\OO_{R\cap \Lambda}(n)) = 3n + 1
\end{equation}
since $R\cap \Lambda$ is a rational normal curve in $\Lambda\cong\PP^3$ for $\Lambda$ general.

Suppose that $R$ is singular at a point $p$. Let $\Lambda$ be a hyperplane passing through $p$. If $R\cap \Lambda$ is an integral curve, then it is a singular rational curve in $\Lambda$ of degree $3$ and hence its Hilbert polynomial is different from \eqref{TRIGCFSE017}, which is a contradiction. So $R\cap \Lambda$ must contain a line. Then $R$ is a cubic cone with vertex at $p$. It is well known that there is a desingularization $f: \widehat{R} \to R$ such that $\widehat{R} = \FF_3$ and $f^* \OO_R(1) = \OO_{\widehat{R}}(B + 3F)$, where $B$ and $F$ are the effective generators of $\Pic(\FF_3)$ with $B^2 = -3$, $BF = 1$ and $F^2 = 0$. Then the Hilbert polynomial of $R$ is
\begin{equation}\label{TRIGCFSE019}
\chi(\OO_{R}(n)) = h^0(\OO_R(n)) = h^0(\OO_{\widehat{R}}(nB + 3nF)) = \dfrac{(3n+2)(n+1)}{2}
\end{equation}
for $n\gg0$.

Consider now the case that $R$ is smooth. Let $p$ be a general point on $R$. The linear system $|\OO_R(1)\otimes I_p^2|$ has dimension at least $1$. By the same argument as before, every member in $|\OO_R(1)\otimes I_p^2|$ contains a line $L$ passing through $p$. The hyperplanes $\Lambda$ passing through $L$ cut out a two-dimensional linear system with base locus $L$. Suppose that
$$
R\cap \Lambda = L + C
$$
with $C$ the moving part. Since $\dim |C| = 2$, $C^2 \ge 1$ on $R$. And since $LC\ge 1$ as $R\cap \Lambda$ is connected, $L^2 \le 0$. Now, since $L$ passes through a general point $p\in R$, $L^2 = 0$. So $|L|$ is a pencil of disjoint lines on $R$ since $H^1(\OO_R) = 0$ as $R$ is obviously rationally connected. That is, the map $R\to \PP^1$ given by $|L|$ realizes $R$ as a rational ruled surface. Then it is easy to see that $R\cong \FF_1$ and $R$ is a rational normal scroll. It is well known that $R$ is the image of the embedding $f: \FF_1\hookrightarrow \PP^4$ by the complete linear series $|B + 2F|$, where $B$ and $F$ are the effective generators of $\Pic(\FF_1)$ with $B^2 = -1$, $BF = 1$ and $F^2 = 0$. Then the Hilbert polynomial of $R$ is the same as \eqref{TRIGCFSE019}.
\end{proof}

So $J_p$ lies on a rational normal scroll or a cubic cone $R_p\subset \operatorname{Bs}(\CQ_p)$. In addition, for such $R_p$, since $h^0(I_{R_p}(2)) = 3$, we have $J_p\subset R_p = \operatorname{Bs}(\CQ_p)$ if $\dim \operatorname{Bs}(\CQ_p) = 2$.

Suppose that $R_p\cong \FF_1$ is a rational normal scroll and
$J_p\in |a B + bF|$, where $B$ and $F$ are effective generators of $\Pic(R_p)$ satisfying $B^2=-1$, $BF=1$ and $F^2 = 1$. Then
$$
b\ge a, \hspace{12pt}
2p_a(J_p) - 2 = 2b(a-1) - a(a+1)\hspace{12pt}\text{and}
\hspace{12pt} \deg J_p = a + b.
$$
When $\tau$ maps $B_p$ birationally onto $J_p$, $p_a(J_p) \ge 5$ and $\deg J_p \le 8$. It is easy to see that $a = 3$ and $b=5$. That is, $J_p\in |3B + 5F|$. In this case, $\deg J_p = 8$ and $p_a(J_p) = 5$. So $J_p = Y_p$ and $\tau$ is regular along $Y_p$.

When $\tau$ maps $B_p$ $2$-to-$1$ onto $J_p$, $J_p\in |B + 3F|$ and $Y_p$ contains the closed subscheme, which we denote by $2J_p$, of $R_p$ with multiplicity $2$ along $J_p$. Then we compare the Hilbert polynomials of $2J_p$ and $Y_p$:
$$
\chi(\OO_{2J_p}(n)) = 8n - 3 > 8n - 4 = \chi(\OO_{Y_p}(n))
$$
which is impossible.

Suppose that $R_p$ is a cubic cone. Let $\pi: \widehat{R}\cong \FF_3 \to R_p$ the minimal resolution of $R_p$ and $\widehat{J}\subset \widehat{R}$ be the proper transform of $J_p$.

Suppose that $\widehat{J}\in |aB + bF|$, where $B$ and $F$ are effective generators of $\Pic(\widehat{R})$ satisfying $B^2=-3$, $BF=1$ and $F^2 = 1$. Then
\begin{equation}\label{TRIGCFSTHMPCME011}
b\ge 3a \hspace{12pt} \text{and} \hspace{12pt}\deg J_p = \widehat{J}(B + 3F) = b.
\end{equation}
The arithmetic genus of $J_p$ can be computed \cite[p. 1044]{CHEN20171033} by computing its Hilbert polynomial
\begin{equation}\label{TRIGCFSTHMPCME012}
\begin{aligned}
&\quad\chi(\OO_{J_p}(n)) = h^0(\OO_{R_p}(n)) - h^0(\OO_{R_p}(n)\otimes I_{J_p}) \\
&= h^0(\OO_{\widehat{R}}(n)) - h^0(\OO_{\widehat{R}}(n) \otimes \OO_{\widehat{R}}(-\widehat{J}))\\
&= h^0(\OO_{\widehat{R}}(n B + 3n F)) - h^0(\OO_{\widehat{R}}((n - a) B + (3n - b) F))\\
&= \dfrac{(3n+2)(n+1)}{2} - \dfrac{1}2 \left(3n+2 - 2b + 3\left\lceil \dfrac{b}3 \right\rceil\right)\left(n+1 - \left\lceil \dfrac{b}3 \right\rceil\right)\\
&= bn + 1 - \left(b - \dfrac{3}2\left\lceil \dfrac{b}3 \right\rceil - 1 \right)\left(\left\lceil \dfrac{b}3 \right\rceil - 1\right)
\end{aligned}
\end{equation}
where $I_{J_p}$ is the ideal sheaf of $J_p$ in $R_p$. It follows that
\begin{equation}\label{TRIGCFSTHMPCME013}
p_a(J_p) = \left(b - \dfrac{3}2\left\lceil \dfrac{b}3 \right\rceil - 1 \right)\left(\left\lceil \dfrac{b}3 \right\rceil - 1\right).
\end{equation}

When $\tau$ maps $B_p$ birationally onto $J_p$, $p_a(J_p) \ge 5$ and $\deg J_p \le 8$. It is easy to see that $b=8$ from \eqref{TRIGCFSTHMPCME011} and \eqref{TRIGCFSTHMPCME013}. 
Hence $\deg J_p = 8$ and $p_a(J_p) = 5$. So $J_p = Y_p$ and $\tau$ is regular along $Y_p$.

When $\tau$ maps $B_p$ $2$-to-$1$ onto $J_p$, $\widehat{J}\in |B + 4F|$ and $Y_p$ contains the closed subscheme, which we denote by $2J_p$, of $R_p$ with multiplicity $2$ along $J_p$. The Hilbert polynomial of $2J_p$ can be computed in the same way as \eqref{TRIGCFSTHMPCME012}, from which we obtain
$$
\chi(\OO_{2J_p}(n)) = 8n - 4 = \chi(\OO_{Y_p}(n)).
$$
This implies that $Y_p = 2J_p$ has pure dimension $1$ and $\tau$ is regular along $Y_p$.

To summarize our above discussion, in the course of proving that $\tau$ is regular, we have also established the facts that 
\begin{itemize}
\item $Y_p$ has pure dimension $1$ for all $p\in C$ and hence $Y$ is Cohen-Macaulay;
\item $Y_p$ has type A1, A2 or A3 for each $p\in C$;
\item $Y_p$ is an integral curve if $Y_p$ has type A1 or A2;
\item $Y_p$ is an integral curve or a closed subscheme with multiplicity $2$ along a rational normal curve on a cubic cone if $Y_p$ has type A3.
\end{itemize}

Thus, in light of these above observations, in order to establish Theorem \ref{TRIGCFSTHMPCM}, it remains to prove that $Y_p$ is reduced if $Y_p$ has type A3. We will argue in two cases:
\begin{itemize}
\item The general fibers of $f$ are non-trigonal.
\item The general fibers of $f$ are trigonal.
\end{itemize}

\subsection{Reducedness of $Y_p$ when $f$ is not a trigonal fibration}

Suppose that the general fibers of $f: X\to C$ is not trigonal. Then the general fibers of $Y/C$ has type A1. We want to show that a special fiber $Y_p$ is not supported on a rational normal curve if $Y_p$ has type A3, i.e., if $Y_p$ lies on the cubic cone $\operatorname{Bs}(\CQ_p)$. Indeed, if $Y_p$ has type A3 and is not reduced, then it must be supported on a rational normal curve. That is, $\tau_* X_p = \tau_* B_p = 2J$ for a rational normal curve $J$ lying on the cubic cone $S = \operatorname{Bs}(\CQ_p)$.

Recall that the hyperplane $\Lambda \in |\OO_W(1)| \cong |\OO_X(8\Gamma)|$ satisfies that
$$
\tau^* \Lambda = 8 \Gamma.
$$
So $\Lambda$ meets $J$ at a unique point $s$, $\tau^{-1}(s) = q = B_p\cap \Gamma$, $\tau: B_p\to J$ is totally ramified over $s$ and $X$ is the normalization of $Y$ locally at $s$.

Let $o$ be the vertex of the cone $S$. If $s\ne o$, then $o\not\in \Lambda$ since $o\in J$ and $s = \Lambda\cap J$.
Then $\Lambda\cap S$ is a smooth curve. The following lemma shows that $\tau^* \Lambda \ne 8 \Gamma$ for such $\Lambda$.
Thus, as a consequence of Lemma \ref{TRIGCFSLEMDOUBLE}, we will deduce that the hyperplane $\Lambda \in |\OO_W(1)| \cong |\OO_X(8\Gamma)|$ meets $J$ at the vertex of the cone $S$.

\begin{Lemma}\label{TRIGCFSLEMDOUBLE}
Let $\tau: X\to R$ be a proper morphism between two analytic varieties with the commutative diagram
$$
\begin{tikzcd}
X \ar{r}{\tau} \ar{rd} & R \ar{d}\\
& \Delta
\end{tikzcd}
$$
such that $X$ and $R$ are smooth over the unit disk $\Delta = \{|t| < 1\}$ of dimensions $\dim X = 2$ and $\dim R = 3$, respectively, $\tau$ is a closed embedding over $\Delta^*$ and $\tau$ maps $X_0$ to a smooth curve $J\subset R_0$ with a map of degree $2$ totally ramified at a point $q\in X_0$. Then there does not exist a hypersurface $\Lambda$ in $R$, smooth over $\Delta$, such that $J\not\subset \Lambda$ and $\tau^* \Lambda$ has multiplicity
\begin{equation}\label{TRIGCFSLEMDOUBLEE003}
\operatorname{mult}_\Gamma \tau^* \Lambda \ge \dfrac{1}2 (\tau^* \Lambda . X_0)_q + 2
\end{equation}
along a section $\Gamma$ of $X/\Delta$ passing through $q$, where $(\tau^* \Lambda . X_0)_q$ is the intersection multiplicity between $\tau^* \Lambda$ and $X_0$ at $q$.
\end{Lemma}

\begin{proof}
We argue by contradiction. First, observe that in an analytic open neighborhood of $p = \tau(q)$, $R\cong \Delta_{xyt}^3$. Set
$J = \{x=t=0\}$. Then, after a change of coordinates, $Z = \tau(X)$ is given by
$$
x^2 = t^a f(y,t)
$$
for some $a\in {\mathbb Z}^+$ and $f(y,t)\in \CC[[y,t]]$ with $f(y,0)\not\equiv 0$.

Since $X_0$ is smooth and $\tau: X_0\to J$ is totally ramified at $q$, we must have that $a$ is even and
$$
f(y,0) = cy + O(y^2)
$$
for some $c\ne 0$, where we use the notation $O(f_1,f_2,...,f_n)$ to denote an element of the ideal generated by $f_1,f_2,...,f_n$ in a ring. So after a change of coordinates, we may assume that $Z$ is given by
$$
x^2 = t^{2n} y.
$$

With these preliminaries at hand, suppose that such $\Lambda$ exists. Let
$$
m = (J . \Lambda)_p
$$
be the intersection multiplicity between $J$ and $\Lambda$ at $p$. Then
$$
(\tau^* \Lambda . X_0)_q = 2m.
$$
So \eqref{TRIGCFSLEMDOUBLEE003} becomes
\begin{equation}\label{TRIGCFSLEMDOUBLEE004}
\operatorname{mult}_\Gamma \tau^* \Lambda \ge m + 2.
\end{equation}
Clearly, $2m \ge m+2$ and hence $m\ge 2$.
And since $\Lambda_0$ is smooth, we see that $\Lambda$ is given by 
$$
x = g(y, t) 
$$
where $g\in \CC[[y,t]]$ satisfies
\begin{equation}\label{TRIGCFSLEMDOUBLEE000}
g(y,0) = by^m + O(y^{m+1})
\end{equation} 
for some $b\ne 0$. Then the condition \eqref{TRIGCFSLEMDOUBLEE004} translates to
\begin{equation}\label{TRIGCFSLEMDOUBLEE001}
(g(y, t))^2- t^{2n} y = (y + \varphi(t))^{m+2}\lambda(y, t)
\end{equation}
for some $\varphi(t)\in \CC[[t]]$ and $\lambda(y,t)\in \CC[[y, t]]$ satisfying $\varphi(0) = 0$.
Then
$$
\begin{aligned}
& \quad (y + \varphi(t))^{m+1} \mid \gcd\big((g(y, t))^2- t^{2n} y, 2 g_y(y,t) g(y,t) - t^{2n}\big)
\\
&\Rightarrow (y + \varphi(t))^{m+1} \mid
\gcd\big( g(y, t) (g(y,t) - 2 g_y(y,t) y), 2 g_y(y,t) g(y,t) - t^{2n} \big)\\
&\Rightarrow (y + \varphi(t))^{m+1} \mid
\gcd\big(g(y,t) - 2 g_y(y,t) y, 2 g_y(y,t) g(y,t) - t^{2n} \big)
\end{aligned}
$$
which is impossible since 
$$g(y, 0) - 2 g_y(y, 0) y = b (1-2m) y^m + O(y^{m+1}).$$
\end{proof}

To apply the above lemma, we see that
$$
\dfrac{1}2 (\tau^* \Lambda . X_0)_q = (\Lambda . J)_s = \deg J = 4
$$
and thus
$$
\operatorname{mult}_\Gamma \tau^* \Lambda < \dfrac{1}2 (\tau^* \Lambda . X_0)_q + 2 = 6
$$
by \eqref{TRIGCFSLEMDOUBLEE003}, which contradicts the fact that $\tau^*\Lambda = 8\Gamma$.
So we necessarily have $s = o$. That is, $\Lambda\cap J = o$, $\tau^{-1}(o) = q$ and $X$ is the normalization of $Y$ locally at $o$.

By \cite[Lemma 4.10, p. 1048]{CHEN20171033},
$Y$ is locally isomorphic to
the surface $\{y^2 = g(x,t)\}\subset\Delta_{xyt}^3$ for some $g(x,t)\in \BC[[x,t]]$ satisfying
\begin{equation}\label{E644}
g(0, t) = t^{\delta} \text{ and }
\left.\frac{\partial g}{\partial x}\right|_{x = 0} \equiv 0,
\end{equation}
at the vertex $o$ of the cone $S$ for some $\delta\in \ZZ^+$. So $Y$ is locally defined by
$$
y^2 = g(x,t) = t^m h(x,t)
$$
for $h(x,t)\in \BC[[x,t]]$ satisfying $h(x,0) \not\equiv 0$. Note that since the general fibers of $Y/C$ are smooth, $h(x,t)$ is reduced. That is,
$$
\gcd(h(x,t), h_x(x,t)) = 1.
$$
So the normalization of $Y$ at $o$ is given by
\begin{equation}\label{TRIGCFSE003}
z^2 = t^{m-2a} h(x,t)\hspace{24pt} \text{for } a = \left\lfloor \dfrac{m}{2}\right\rfloor
\end{equation}
in $\Delta_{xt}^2\times \BA_z^1$. Since $X$ is the normalization of $Y$ locally at $o$, $X$ is locally given by \eqref{TRIGCFSE003} over $o$. However, this is impossible since
\begin{itemize}
\item if $m$ is odd, then $X_p = \{z^2 = 0\}\subset \Delta_{x}^1 \times \BA_z^1$ is non-reduced;
\item if $m$ is even and $h(0,0) \ne 0$, then $\tau^{-1}(o)$ consists of two distinct points;
\item if $m$ is even and $h(0,0) = 0$, then $X_p = \{z^2 = h(x,0)\}\subset \Delta_{x}^1\times \BA_z^1$ is singular at $q = \tau^{-1}(o) = \{x = z = t = 0\}$ since
$h_x(0,0) = 0$ by \eqref{E644}.
\end{itemize}

\subsection{Reducedness of $Y_p$ when $f$ is a trigonal fibration}

Suppose that the general fibers of $f: X\to C$ are trigonal. Then each fiber $Y_p$ is of type either A2 or A3. Furthermore, there exists a closed subvariety
$$
M = \bigcup_{p\in C} \operatorname{Bs}(\CQ_p)\subset W,
$$ 
such that $Y \subset M$ and a general fiber $M_p$ of $M$ is a rational normal scroll in $W_p$. Note that $M$ is flat over $C$ by Lemma \ref{TRIGCFSLEMCUBIC}.

It is easy to see the following observation which is of some independent interest and also is important for our purposes here:

\begin{Lemma}\label{TRIGCFSLEMCUBICFIB}
Let $W$ be a $\PP^4$-bundle over a smooth curve $C$ and let $M\subset W$ be a closed subvariety of $W$, flat over $C$, such that the fiber $M_p$ is a rational normal scroll in $W_p$ for $p\in C$ general. Suppose that $M_p$ contains a linearly non-degenerate curve for every $p\in C$. Then
\begin{itemize}
\item $M_p$ is either a rational normal scroll or a cubic cone for every $p\in C$, and
\item there exists a desingularization $\xi: R\to M$ with commutative diagram
$$
\begin{tikzcd}
R \ar{r}{\xi} \ar{d}[left]{\phi} & M\ar{d}\\
S \ar{r}{\pi} & C
\end{tikzcd}
$$
such that $R$ and $S$ are $\PP^1$-bundles over $S$ and $C$, respectively, $R_p = \xi^{-1}(M_p)$ is either $\FF_1$ or $\FF_3$, and the map $\xi: R_p \xrightarrow{} M_p$ is either an isomorphism if $R_p\cong \FF_1$ or the contraction of the $(-3)$-curve if $R_p\cong \FF_3$ for each $p\in C$.

In particular, $R^\bullet \xi_* \OO_R = \OO_M$ and $M$ has rational singularities and is hence normal and Cohen-Macaulay.
\end{itemize}
\end{Lemma}

\begin{proof}
Clearly, $\deg M_p = 3$ for all $p\in C$. So $M_p$ is supported on either an integral cubic surface, a union of a quadric and plane or a union of three planes. And since $M_p$ contains a linearly non-degenerate curve, it must be supported on an integral linearly non-degenerate cubic surface. Thus, by Lemma \ref{TRIGCFSLEMCUBIC}, $S = \supp(M_p)$ is either a rational normal scroll or a cubic cone. By the flatness of $M$ over $C$, the Hilbert polynomial of $M_p$ agrees with that of a rational normal scroll and thus that of $S$. That is,
$$
h^0(\OO_{M_p}(n)) = h^0(\OO_S(n))
$$
for $n$ sufficiently large. This implies that the kernel of $\OO_{M_p} \to \OO_S$ is zero. That is, $S = M_p$ and $M_p$ is integral without any embedded component.

Let us consider over $M$ the relative Fano variety of lines
$$
S' = \big\{(p, [L]): p\in C, [L]\in {\mathbb{G}\text{r}}(1, W_p), L \subset M_p\big\}
$$
where ${\mathbb{G}\text{r}}(1, W_p)$ is the Grassmannian of lines in $W_p\cong \PP^4$. When $M_p$ is a rational normal scroll, $S_p'$ has two disjoint components: one is a smooth rational curve and the other is a point. So $S'$ has two irreducible components with relative dimension $1$ and $0$ over $C$, respectively. We let $S$ be the component of $S'$ of dimension $2$. Clearly, $S$ is a $\PP^1$-bundle over $C$.

Let
$$
R = \{(p, [L], q): (p, [L])\in S, q\in L\}\subset S\times_C W
$$
be the universal family over $S$. Clearly, the projection $\xi: R\to W$ maps $R$ to $M$, which resolves the singularities of $M$ with the properties stated in the lemma.
\end{proof}

Let $\xi: R\to W$ be the desingularization of $M$ given in Lemma \ref{TRIGCFSLEMCUBICFIB}. Then we have the diagram
$$
\begin{tikzcd}
& R\ar{d}{\xi} \ar{r}{\phi} & S \ar{d}{\pi}\\
X \ar[dashed]{ur}{\psi}\ar{r}{\tau} & W \ar{r} & C
\end{tikzcd} 
$$
where we let $Z = \psi(X)$ be the proper transform of $X$ under $\psi$ and $Q = \xi^* \Lambda$ with $\Lambda \in |\OO_W(1)| \cong |\OO_X(8\Gamma)|$ being the unique section. As before, we set $Y = \tau(X)$.

Obviously, $Y_p$ is reduced in the A2 case. We thus need to prove that $Y_p$ is reduced when $M_p$ is a cubic cone, i.e., 
$R_p\cong \FF_3$. If this is not the case, then $Y_p$ is a closed subscheme in $M_p$ and supported on a smooth rational curve with multiplicity $2$. It is easy to see that, in this case,
$$
Z_p = B + 2J
$$
on $R_p \cong \FF_3$, where $J\in |B + 4F|$ is a smooth rational curve whose image under $\xi: R\to W$ is a rational normal curve.

To prove that this cannot happen, and thus that $Y_p$ is reduced in the A3 case, it suffices to prove the proposition below:

\begin{Proposition}\label{TRIGCFSPROP003}
Let $\psi: X/\Delta\dashrightarrow R/\Delta$ be a rational map over the unit disk $\Delta$, where
\begin{itemize}
\item $X$ is a flat projective family of curves over $\Delta$;
\item $X$ is smooth and $g(X_t) = 5$ for $t\ne 0$;
\item $R$ is a smooth projective family of surfaces over $\Delta$ with $R_0\cong \FF_3$ and $R_t\cong \FF_1$ for $t\ne 0$;
\item $\psi$ is a regular closed embedding over $\Delta^*$ and $\psi_* X_t\in |3B + 5F|$
for $t\ne 0$, where $B$ and $F$ are effective generators of $\Pic(\FF_1)$ and satisfying the conditions that $B^2 = -1$, $BF = 1$ and $F^2 = 0$.
\end{itemize}
Suppose that there exist a section $\Gamma\subset X$ of $X/\Delta$ and an effective divisor $Q\subset R$ such that
\begin{itemize}
\item $Q$ is flat over $\Delta$, $Q_t\in |B + 2F|$ for $t\ne 0$, and
\item $\psi^* Q = 8\Gamma$ over $\Delta^*$.
\end{itemize}
Then there does not exist an irreducible component $M$ of $X_0$ such that 
$M\cap \Gamma \ne \emptyset$ and $\psi$ maps $M$ $2$-to-$1$ onto an integral curve $J\in |B + 4F|$, where $B$ and $F$ are effective generators of $\Pic(R_0) = \Pic(\FF_3)$ and satisfying the conditions that $B^2 = -3$, $BF=1$ and $F^2 = 0$.
\end{Proposition}

In addition to Lemma \ref{TRIGCFSLEMDOUBLE}, we need a few more ``local'' lemmas, in order to prove Proposition \ref{TRIGCFSPROP003}. They are formulated as Lemmas \ref{TRIGCFSLEM000}, \ref{TRIGCFSLEM001} and \ref{TRIGCFSLEM002} below.

\begin{Lemma}\label{TRIGCFSLEM000}
Let $\psi: X/\Delta\to R/\Delta$ be a proper morphism between analytic varieties over the unit disk $\Delta = \{|t| < 1\}$, where $X\cong \{uv = t^m\}\subset \Delta_{uvt}^3$ for some $m\in \ZZ^+$ and $R \cong \Delta_{xyt}^3$. Suppose that $\psi$ is a closed embedding over $\Delta^*$. If the restriction $\psi: X_0\to R_0$ of $f$ to $t=0$ is given by $\psi(u,v,0) = (u^a, v, 0)$ for some $a\in \ZZ^+$, then 
there exists an automorphism $g\in \Aut(R/\Delta)$ such that
$g\circ \psi(X)$ is cut out by
\begin{equation}\label{TRIGCFSLEM000E000}
x\left(y^a + \lambda(x,y,t) \right) = t^{am} + \delta(y,t)
\end{equation}
on $R$ for some $\lambda(x,y,t)\in \CC[y, [x,t]]$
and $\delta(y,t)\in \CC[y, [t]]$
satisfying
$$
\begin{aligned}
\deg_y \lambda(x,y,t) \le a-2, &\hspace{6pt} \deg_y \delta(y,t) \le a-1,\hspace{6pt} \lambda(x,0,t)\equiv \delta(0,t) \equiv 0\\
t^{am} \mid \lambda(x,t^m y, t)&\hspace{6pt} \text{and}\hspace{6pt} t^{am + 1} \mid \delta(t^my, t).
\end{aligned}
$$
In particular, when $a = 2$, there exists an automorphism $g\in \Aut(R/\Delta)$ such that
$g\circ \psi(X)$ is cut out by
\begin{equation}\label{TRIGCFSLEM000E001}
x y^2  = t^{2m} \hspace{6pt}\text{or}\hspace{6pt} xy^2 = t^{2m} + t^{m+c} y \hspace{12pt} \text{for some } 0 < c < m
\end{equation}
on $R$ and there exists an automorphism $h\in \Aut(X/\Delta)$ such that
\begin{equation}\label{TRIGCFSLEM000E002}
g\circ \psi\circ h(u,v,t) = (u^2, v, t) \hspace{6pt}\text{or}\hspace{6pt}
g\circ \psi\circ h(u,v,t) = (u^2 + t^c u, v, t).
\end{equation}
\end{Lemma}

\begin{proof}
There is nothing to prove for $a=1$. Let us assume that $a\ge 2$.

Let $Z = \psi(X)$. Clearly, $Z_0 = \{xy^a = 0\} \subset R_0\cong \Delta_{xy}^2$. Using the Weierstrass Preparation Theorem, we can put the defining equation of $Z$ in the following {\em standard} form:
\begin{equation}\label{TRIGCFSLEM000E003}
\begin{aligned}
& (x+\alpha(y,t)) \left(y^a + \beta(x,y,t) \right) \theta(x,y,t) = \gamma(y,t)
\\
&\hspace{12pt} \text{for some } \alpha(y,t)\in \CC[[y,t]],\hspace{6pt} \beta(x,y,t)\in \CC[y,[x,t]],\\
&\hspace{54pt}
\gamma(y,t)\in \CC[y,[t]] \hspace{6pt}\text{and}\hspace{6pt}\theta(x,y,t)\in \CC[[x,y,t]]\\
&\hspace{18pt}\text{satisfying that } \alpha(y,0)\equiv \beta(x,y,0) \equiv \gamma(y,0)\equiv 0,\ \theta(0,0,0) \ne 0,\\
&\hspace{90pt} \deg_y \beta(y,t) \le a-1 \hspace{6pt}\text{and}\hspace{6pt} \deg_y \gamma(y,t) \le a-1.
\end{aligned}
\end{equation}

We can find an automorphism $g\in \Aut(R/\Delta)$, or equivalently, a coordinate change, such that $g(Z)$ is given by
\begin{equation}\label{TRIGCFSLEM000E004}
x(y^a + \lambda(x,y,t)) = \mu(y,t) 
\end{equation}
for some $\lambda(x,y,t)\in \CC[y, [x,t]]$ and $\mu(y,t)\in \CC[y, [t]]$ satisfying the condition that
$$\lambda(x,y,0) \equiv \mu(y,0)\equiv 0,\ \deg_y \lambda(x,y,t)\le a -2 \text{ and } \deg_y \mu(y,t) \le a-1.$$
For simplicity, we replace $\psi$ by $g\circ \psi$ and assume that $Z$ is given by \eqref{TRIGCFSLEM000E004} in $R$.

For $f(x,y,t)\in \CC[[x,y,t]]$, we define $m=\nu(f)$ to be largest integer such that $t^m \mid f(x,y,t)$. We extend this definition to
$f\in \CC[[x,y,\sqrt[n]{t}]]$ for which $\nu(f)\in \QQ_{\ge 0}$.

We let $d$ be the largest rational number such that
$$
ad = \min(\nu(\lambda(x,t^d y, t)), \nu(\mu(t^d y, t)))
$$
After a base change $\Delta\to \Delta$, we may assume that $d\in \ZZ^+$. Let $\pi: \widehat{R}\to R$ be the blowup of $R$ along the closed subscheme $y = t^d = 0$. We have the diagram
$$
\begin{tikzcd}
& X\ar[dashed]{d}{\phi} \ar[dashed]{ld} \ar{rd}{\psi}\\
\widehat{Z} \ar[hook]{r}& \widehat{R} \ar{r}{\pi} & R
\end{tikzcd}
$$
where $\widehat{Z}$ is the proper transform of $Z$ under $\pi$.

Let $C_1 = \{x = t = 0\}$ and $C_2 = \{y = t= 0\}$. 
The central fiber $\widehat{R}_0$ of $\widehat{R}$ is a union $\widehat{R}_0 = S\cup E$, where $S\cong R_0$ is the proper transform of $R_0$ and $E\cong \PP^1\times C_2$ is the exceptional divisor of $\pi$.
We write $\widehat{Z}_0 = \widehat{C}_1 + \widehat{C}_2$, where $\widehat{C}_1\subset S$ is the proper transform of $C_1$ and $\widehat{C}_2\subset E$.

We can write down the defining equation of $\widehat{C}_2$ in $E\cong \PP^1\times C_2$ explicitly:
\begin{equation}\label{TRIGCFSLEM000E005}
\begin{aligned}
x(w^a + f(x,w)) &= h(w)\hspace{12pt}\text{for}\\
f(x,w) &= \left.\dfrac{\lambda(x,t^d w, t)}{t^{ad}}\right|_{t=0}
\hspace{12pt}\text{and}\\
h(w) &= \left.\dfrac{\mu(t^d w, t)}{t^{ad}}\right|_{t=0}
\end{aligned}
\end{equation}
where $w$ is the affine coordinate of $E\backslash S \cong \BA^1\times C_2$ with $y = t^d w$. Note that $f(x,w)\in \CC[w,[x]]$, $h(w)\in \CC[w]$, $\deg_w f(x,w) \le a-2$, $\deg_w h(w) \le a-1$ and at least of one of $f(x,w)$ and $h(w)$ is nonzero by our choice of $d$.

Since $X$ is the stable reduction of $Z$, there do not exist two distinct irreducible components of $\widehat{C}_2$ dominating $C_2$ via $\pi$. Combining this with the fact that $\deg_w f(x,w)\le a -2$, we see that
\begin{itemize}
\item $w^a + f(x,w)$ is irreducible in the ring $\CC[w,[x]]$ if $h(w)\equiv 0$;
\item $x(w^a + f(x,w)) - h(w)$ is irreducible in the ring $\CC[w,[x]]$ if $h(w)\not\equiv 0$.
\end{itemize}
Geometrically, there exits an irreducible component $\Gamma\subset \widehat{C}_2$, defined by either $w^a + f(x,w) = 0$ or $x(w^a + f(x,w)) = h(w)$ such that $\pi_* \Gamma = a C_2$.

We claim that $h(w) \not\equiv 0$. Otherwise,
$$
\widehat{Z}_0 = \widehat{C}_1 + \widehat{C}_2 = \widehat{Z}_0 = \widehat{C}_1 + G + \Gamma
$$
where $G = \{x=0\}\subset E$, $G\cong \PP^1$, $\Gamma = \{w^a + f(x,w) = 0\} \subset E$ and $\widehat{C}_1\cap \Gamma = \emptyset$. 

Since $\psi$ is a closed embedding over $\Delta^*$, $Z$ is smooth over $\Delta^*$. The same holds for $\widehat{Z}$ and since $\widehat{Z}_0$ is reduced, $\widehat{Z}$ is smooth in codimension one. Also since $\widehat{Z}$ is a Cartier divisor in $\widehat{R}$, it is a local complete intersection and we deduce that $\widehat{Z}$ is Cohen-Macaulay. So by Serre's criterion, $\widehat{Z}$ is smooth.

After a base change, let $\xi: \widehat{X}\to \widehat{Z}$ be the stable reduction of $\widehat{Z}$. We have the diagram
$$
\begin{tikzcd}
\widehat{X} \ar{r} \ar{d}{\xi} & X\ar[dashed]{d}{\phi} \ar[dashed]{ld} \ar{rd}{\psi}\\
\widehat{Z} \ar[hook]{r}& \widehat{R} \ar{r}{\pi} & R
\end{tikzcd}
$$
Clearly, $\widehat{X}_0 = D_1 + D_2 + D_3$ is the union of three smooth curves $D_i$ such that $D_1D_2=1$, $D_2D_3 = 1$, $D_3D_1 = 0$, $\xi_* D_1 = \widehat{C}_1$, $\xi_* D_2 = G$ and $\xi_* D_3 = \Gamma$. So $\xi: \widehat{X}\to \widehat{Z}$ is birational and finite. And since both $\widehat{X}$ and $\widehat{Z}$ are normal, $\xi$ must be an isomorphism and hence $G$ and $\Gamma$ meet transversely at a unique point. But $G\Gamma = a\ge 2$ on $E$, which is a contradiction. So $h(w)\not\equiv 0$.

Consequently, $\widehat{C}_2$ is an integral curve defined by \eqref{TRIGCFSLEM000E005} and mapped to $C_2$ $m$-to-$1$ by $\pi$. So $\pi: \widehat{Z}\to Z$ is birational and finite.
Thus, $\phi: X\dashrightarrow \widehat{Z}$ is a regular morphism and actually an isomorphism since $\widehat{Z}$ is normal by the same argument as above. From this, we conclude the following:
\begin{itemize}
\item The map $\pi: \widehat{C}_2\to C_2$ is totally ramified at the point $q = \widehat{C}_1\cap \widehat{C}_2$. Using the defining equation \eqref{TRIGCFSLEM000E005} of $\widehat{C}_2$, we see that $h(w) \equiv h(0) \ne 0$ and hence
$$
\mu(y,t) = h(0) t^{ad} + \delta(y, t)
$$
for $\delta(y,t)\in \CC[y,[t]]$ satisfying that $\delta(0,t)\equiv 0$ and $t^{ad+1}\mid \delta(t^d y, t)$.
\item Since $\phi: X\to \widehat{Z}$ is an isomorphism and it is easy to see that
$$\widehat{Z}\cong \{yz = t^d\}\subset \Delta_{yzt}^3$$
at $q$, we conclude that $d = m$.
\end{itemize}
Combining the above, we see that $Z$ is given by
$$
x\left(y^a + \lambda(x,y,t) \right) = t^{am} + \delta(y,t)
$$
after a coordinate change. By our choice of $d$, $t^{ad} = t^{am}\mid \lambda(x, 0, t)$. So the term $x\lambda(x,0,t)$ can be ``absorbed'' into $t^{am}$. That is, after a further coordinate change, we obtain \eqref{TRIGCFSLEM000E000}.

When $a=2$, $\lambda(x,y,t) \equiv 0$ since
$$\deg_y \lambda(x,y,t) \le a-2 = 0\hspace{24pt} \text{and} \hspace{24pt} \lambda(x,0,t)\equiv 0.$$
So we obtain \eqref{TRIGCFSLEM000E001}. Then we normalize $Z$ and \eqref{TRIGCFSLEM000E002} follows easily.
\end{proof}

In the above lemma, the hypothesis that $\psi: X\to R$ is a closed embedding over $\Delta^*$ is crucial. Otherwise, the statement is false.

\begin{Lemma}\label{TRIGCFSLEM001}
Let $Z$ and $Q$ be analytic subvarieties of $R \cong \Delta_{xyt}^3$ given by
\begin{equation}\label{TRIGCFSLEM001E000}
Z = \big\{x y^a = t^{am}\big\} 
\end{equation}
and
\begin{equation}\label{TRIGCFSLEM001E001}
Q = \left\{x\prod_{i=1}^{n-1} (x+ b_i y) + O(t, x^{n+1}, x^n y,...,x y^{n}) = 0\right\},
\end{equation}
respectively, for some $a, m, n\in \ZZ^+$ and $b_1,b_2,...,b_{n-1}\in \CC$ satisfying the condition that $a, n\ge 2$ and
\begin{equation}\label{TRIGCFSLEM001E002}
b_1 + b_2 + ... + b_{n-1}\ne 0,
\end{equation}
where $O(t, x^{n+1}, x^n y,...,x y^{n})$ is an element of the ideal in $\CC[[x,y,t]]$ generated by $t, x^{n+1}, x^n y,...,x y^{n}$.
Then
\begin{equation}\label{TRIGCFSLEM001E011}
Z.Q = d D + \Sigma
\end{equation}
for $D = \{x = t = 0\}$, some $d\in \ZZ^+$ and some (nonempty) curve $\Sigma \subset R$ flat over $\Delta = \{|t| < 1\}$ and there does not exists a section $\Gamma$ of $R/\Delta$ such that $\Sigma = an \Gamma$.
\end{Lemma}

\begin{proof}
Let $f(x,y,t) = 0$ be a defining equation of $Q$. By the Weierstrass Preparation Theorem, we may choose $f(x,y,t)$ to be a monic polynomial in $\CC[[y,t]][x]$ of degree $n$ in $x$. In addition, we may replace $Q$ by $$Q' = \{f(x,y,t) + (xy^a - t^{am})g(x,y,t) = 0\}$$
for any $g(x,y,t)\in \CC[[x,y,t]]$ since $Z.Q = Z.Q'$. 

For every $\varphi(x,y,t)\in \CC[[x,y,t]]$, there exists $g(x,y,t)\in \CC[[x,y,t]]$ such that
\begin{equation}\label{TRIGCFSLEM001E006}
\varphi(x,y,t) + (xy^a - t^{am}) g(x,y,t)
= \sum_{p=1}^\infty \sum_{q=0}^{a-1} \varphi_{p,q}(t) x^p y^q
+ \varphi_0(y,t)
\end{equation}
for some $\varphi_{p,q}(t)\in \CC[[t]]$ and $\varphi_0(y,t)\in \CC[[y,t]]$. In short, the right hand side of \eqref{TRIGCFSLEM001E006} does not have terms $x^p y^q$ for $p\ge 1$ and $q\ge a$.

In other words, we can choose the right hand side of \eqref{TRIGCFSLEM001E006} as a representation of $\varphi(x,y,t)$ in the ring $\OO_Z = \CC[[x,y,t]]/(xy^a - t^{am})$. In addition, this representation is unique. So we can define a $\CC$-linear map
$\sigma: \CC[[x,y,t]] \to \CC[[x,y,t]]$ by
\begin{equation}\label{TRIGCFSLEM001E012}
\sigma(\varphi(x,y,t))
= \sum_{p=1}^\infty \sum_{q=0}^{a-1} \varphi_{p,q}(t) x^p y^q
+ \varphi_0(y,t).
\end{equation}
After we replace $f(x,y,t)$ by $\sigma(f(x,y,t))$, we have
\begin{equation}\label{TRIGCFSLEM001E005}
\begin{aligned}
f(x,y,t) &= x^n + b x^{n-1}y + \sum_{p=0}^{n-1}\sum_{q=0}^\infty
f_{p,q}(t) x^p y^q 
\\
\text{where } b &= b_1 + b_2 + ... + b_{n-1}\ne 0,\ f_{p,q}(t) \in \CC[[t]],\\
f_{p,q}(t) &\equiv 0 \hspace{6pt}\text{for all } p\ge 1 \text{ and } q\ge a, \hspace{12pt}\text{and}\\
f_{p,q}(0) &= 0 \hspace{6pt}\text{for } p+q < n \text{ or } (p,q) = (n-1,1) \text{ or } p=0.
\end{aligned}
\end{equation}

Let $\psi: X\to R$ be the normalization of $Z$. Clearly,
$$X \cong \{uv=t^m\}\subset \Delta_{uvt}^3 \hspace{12pt}\text{and}\hspace{12pt} \psi(u,v,t) = (u^a, v,t).$$
Let
\begin{equation}\label{TRIGCFSLEM001E003}
\begin{aligned}
g(v,t) &= v^{an} \psi^* (f(x,y,t)) = v^{an} f(u^a, v, t)
\\
&= v^{an}\big(u^{an} + b u^{a(n-1)}v + \sum_{p=0}^{n-1}\sum_{q=0}^\infty
f_{p,q}(t) u^{ap} v^q\big)\\
&= t^{amn} + b t^{am(n-1)} v^{a+1} + \sum_{p=0}^{n-1}\sum_{q=0}^\infty
f_{p,q}(t) t^{amp} v^{a(n-p) + q}
\end{aligned}
\end{equation}

We let $\delta\in \QQ_{\ge 0}$ be the smallest number such that
$$
\nu(g(t^\delta v, t)) = amn
$$
where $\nu(g)$ is defined as before to be the largest number $l$ such that $t^l \mid g$.
That is,
$$
\delta = \max\left\{\dfrac{am}{a+1}\right\} \cup \left\{\dfrac{am(n-p) - \nu(f_{p,q}(t))}{a(n-p)+q}: 0\le p<n, q\ge 0\right\}
$$
where we let $\nu(0) = \infty$.
After a base change, we may assume that $\delta\in \ZZ$.
Clearly, $0<\delta < m$.

We let
$$
h(w, t) = \dfrac{g(t^\delta w, t)}{t^{amn}}.
$$
It is easy to see from \eqref{TRIGCFSLEM001E003} that
\begin{equation}\label{TRIGCFSLEM001E004}
h(w, 0) = 1 + w^a \phi(w)
\end{equation}
for some $\phi(w)\in \CC[w]$. Due to our choice of $\delta$, $\phi(w) \not\equiv 0$. So there exists $w_0\in \CC^*$ such that $h(w_0, 0) = 0$. It follows that there exist $e\in \ZZ^+$ and $\gamma(t)\in \CC[[\sqrt[e]{t}]]$ such that $\gamma(0) = w_0$ and $g(t^\delta \gamma(t), t) \equiv 0$. In other words, $\beta(t) = t^\delta \gamma(t)$ satisfies that $0<\nu(\beta(t))<m$ and $g(\beta(t),t)\equiv 0$. This implies that $\psi^{-1}(Q)$ contains a component flat over $\Delta$ and it proves that $\Sigma \ne \emptyset$ in \eqref{TRIGCFSLEM001E011}.

Suppose that there is a section $\Gamma$ of $R/\Delta$ satisfying 
$$
Z . Q = d D + an\Gamma
$$
This is equivalent to saying that
$$
\psi^* Q = dC_u + an \Gamma_X
$$
where $C_u = \{u = t=0\}$ and $\Gamma_X = \psi^{-1}(\Gamma)$. So $h(w,0)$ has a unique zero of multiplicity $an$. That is,
$$
h(w,0) = c(w - w_0)^{an}
$$
for some $c\ne 0$. Clearly, this is impossible by \eqref{TRIGCFSLEM001E004} since $a\ge 2$.
\end{proof}

\begin{Lemma}\label{TRIGCFSLEM002}
Let $Z$ and $Q$ be analytic subvarieties of $R \cong \Delta_{xyt}^3$ given by
\begin{equation}\label{TRIGCFSLEM002E000}
Z = \big\{x y^2 = t^{2m} + t^{m+c} y\big\} 
\end{equation}
and
\begin{equation}\label{TRIGCFSLEM002E001}
Q = \Big\{f(x,y,t) = x\prod_{i=1}^{n-1} (x+ b_i y) + O(t, x^{n+1}, x^n y,...,x y^{n}) = 0\Big\},
\end{equation}
respectively, for some $c, m, n\in \ZZ^+$ and $b_1,b_2,...,b_{n-1}\in \CC$ satisfying that $m > c$, $n\ge 2$ and
\begin{equation}\label{TRIGCFSLEM002E002}
b_1 + b_2 + ... + b_{n-1}\ne 0,
\end{equation}
where $O(t, x^{n+1}, x^n y,...,x y^{n})$ is an element of the ideal in $\CC[[x,y,t]]$ generated by $t, x^{n+1}, x^n y,...,x y^{n}$.
Then
\begin{equation}\label{TRIGCFSLEM002E012}
Z.Q = d D + \Sigma
\end{equation}
for $D = \{x = t = 0\}$, some $d\in \ZZ^+$ and some (nonempty) curve $\Sigma \subset R$ flat over $\Delta = \{|t| < 1\}$ and there does not exists a section $\Gamma$ of $R/\Delta$ such that $\Sigma = 2n \Gamma$.
\end{Lemma}

\begin{proof}
The argument is similar to that of Lemma \ref{TRIGCFSLEM001}.
As before, we can define a $\CC$-linear map
$\sigma: \CC[[x,y,t]]\to \CC[[x,y,t]]$ as in \eqref{TRIGCFSLEM001E012} with respect to the ideal $(xy^2 - (t^{2m} + t^{m+c} y))$. That is, for every $\varphi(x,y,t)\in \CC[[x,y,t]]$, we let $\sigma(\varphi(x,y,t))$ be the (unique) power series such that
\begin{equation}\label{TRIGCFSLEM002E004}
\begin{aligned}
&\sigma(\varphi(x,y,t)) - \varphi(x,y,t) \in (xy^2 - (t^{2m} + t^{m+c} y)), \hspace{12pt} \text{and}
\\
& \sigma(\varphi(x,y,t)) 
= \sum_{p=1}^\infty \sum_{q=0}^{1} \varphi_{p,q}(t) x^p y^q
+ \varphi_0(y,t)
\end{aligned}
\end{equation}
for some $\varphi_{p,q}(t)\in \CC[[t]]$ and $\varphi_0(y,t)\in \CC[[y,t]]$.

After we apply Weierstrass Preparation Theorem to $f(x,y,t)$ and replace it by $\sigma(f(x,y,t))$, we obtain
\begin{equation}\label{TRIGCFSLEM002E005}
\begin{aligned}
f(x,y,t) &= x^n + b x^{n-1}y + \sum_{p=0}^{n-1}\sum_{q=0}^\infty
f_{p,q}(t) x^p y^q 
\\
\text{where } b &= b_1 + b_2 + ... + b_{n-1}\ne 0,\ f_{p,q}(t) \in \CC[[t]],\\
f_{p,q}(t) &\equiv 0 \hspace{6pt}\text{for all } p\ge 1 \text{ and } q\ge 2, \hspace{12pt}\text{and}\\
f_{p,q}(0) &= 0 \hspace{6pt}\text{for all }p,q.
\end{aligned}
\end{equation}

Let $\psi: X\to R$ be the normalization of $Z$. Clearly,
$$X \cong \{uv=t^m\}\subset \Delta_{uvt}^3 \hspace{12pt}\text{and}\hspace{12pt} \psi(u,v,t) = (u^2 + t^c u, v,t).$$
Let us consider
\begin{equation}\label{TRIGCFSLEM002E010}
\begin{aligned}
g(v,t) &= v^{2n} \psi^* (f(x,y,t)) = v^{2n} f(u^2 + t^c u, v, t)
\\
&= v^{2n}\big((u^2 + t^c u)^{n} + b (u^2 + t^c u)^{n-1} v\\
&\hspace{36pt} + \sum_{p=0}^{n-1}\sum_{q=0}^\infty
f_{p,q}(t) (u^2 + t^c u)^{p} v^q\big)\\
&= (t^{2m} + t^{m+c} v)^{n} + b (t^{2m} + t^{m+c} v)^{n-1} v^3\\
&\hspace{12pt} + \sum_{p=0}^{n-1}\sum_{q=0}^\infty
f_{p,q}(t) (t^{2m} + t^{m+c} v)^{p} v^{2(n-p) + q}\\
&= t^{2mn} + \sum_{k=1}^\infty g_k(t) v^k \hspace{24pt} \text{for some } g_k(t)\in \CC[[t]].
\end{aligned}
\end{equation}
As before, we see that $\psi^{-1}(Q)$ has a component flat over $\Delta$ if and only if there exist $e\in \ZZ^+$ and $\beta(t)\in \CC[[\sqrt[e]{t}]]$ such that $0 < \nu(\beta(t)) < m$ and $g(\beta(t),t)\equiv 0$.

Let
\begin{equation}\label{TRIGCFSLEM002E009}
\delta = \max \left\{\dfrac{2mn - \nu(g_k(t))}{k}: k\ge 1\right\}
\end{equation}
where $\nu(g)$ is defined as before to be the largest number $l$ such that $t^l \mid g$.

After a base change, we may assume that $\delta\in \ZZ$. 
Clearly, 
\begin{equation}\label{TRIGCFSLEM002E017}
\nu(g_1(t)) = 2m(n-1) + m+c
\end{equation}
and hence $\delta \ge m-c$. And since $\nu(g(t^m v, t) - t^{2mn}) > 2mn$,
$\delta < m$. Thus,
$$
m-c \le \delta < m.
$$ 
Let
$$
h(w,t) = \dfrac{g(t^\delta w, t)}{t^{2mn}}.
$$
Obviously, $h(w,0)\in \CC[w]$. By our choice of $\delta$,
$$h(0, 0)\ne 0 \hspace{12pt}\text{and}\hspace{12pt} \deg_w h(w,0) \ge 1.$$
So there exists $w_0\ne 0$ such that $h(w_0, 0) = 0$. Consequently, there exist $e\in \ZZ^+$ and $\gamma(t)\in \CC[[\sqrt[e]{t}]]$ such that $\gamma(0) = w_0$ and $h(\gamma(t), t)\equiv 0$. That is, $g(\beta(t), t) \equiv 0$ for $\beta(t) = t^\delta \gamma(t)$ with $0 < \nu(\beta(t)) = \delta < m$. Hence $\psi^{-1}(Q)$ has a component flat over $\Delta$.
This proves that $\Sigma\ne \emptyset$ in \eqref{TRIGCFSLEM002E012}.

Suppose that there exists a section $\Gamma$ of $R/\Delta$ such that
$$
Z.Q = d D + 2n\Gamma
$$
or equivalently,
\begin{equation}\label{TRIGCFSLEM002E008}
\psi^* Q = d C_u + 2n \Gamma_X
\end{equation}
where $C_u = \{u=t=0\}$ and $\Gamma_X = \psi^{-1}(\Gamma)$ are the proper transforms of $D$ and $\Gamma$ under $\psi$, respectively. Then $h(w,0)$ has a unique root of multiplicity $2n$.

If $\delta > m-c$, then by \eqref{TRIGCFSLEM002E017},
\begin{equation}\label{TRIGCFSLEM002E006}
h(w,0) = 1 + w^2 \lambda(w)
\end{equation}
for some polynomial $\lambda(w)\not\equiv 0\in \CC[w]$. Clearly, $h(w,0)$ has at least two distinct roots in $\CC^*$. So $\delta = m - c$ and
\begin{equation}\label{TRIGCFSLEM002E011}
\begin{aligned}
\Gamma_X &= \big\{u = \alpha(t), v = \beta(t)\big\}\subset X\\
&\text{for some } \alpha(t), \beta(t)\in \CC[[t]] \text{ satisfying that}\\
&\hspace{24pt} \alpha(t)\beta(t)\equiv t^m \hspace{12pt}\text{and}\hspace{12pt}\nu(\alpha(t)) = c.
\end{aligned}
\end{equation}

Note that $(v-\beta(t))^{2n} \mid g(v,t)$. We claim that
\begin{equation}\label{TRIGCFSLEM002E003}
g(v,t) = (t^m - \alpha(t) v)^{2n} \theta(v,t)
\end{equation}
for some $\theta(v,t)\in \CC[[v,t]]$ with $\theta(0,0) = 1$.
This follows from the Weierstrass Preparation Theorem if
\begin{equation}\label{TRIGCFSLEM002E018}
\ell = \min \big\{\nu(g_k(t)): k\ge 1\big\} \ge 2nc
\end{equation}
Let us prove \eqref{TRIGCFSLEM002E018}. Otherwise, suppose by contradiction that $\ell < 2nc$. Let $a$ be the smallest positive integer such that $\nu(g_a(t)) = \ell$.

It is clear that
$$
\nu(g_k(t)) = 2mn - (m-c)k
$$
for $1\le k \le 2n$. So $a > 2n$ and $\nu(g_k(t)) > \nu(g_a(t))$ for all $k < a$. Let
$$
\mu = \min \left\{\dfrac{2mn-\ell}{a}\right\}\cup \left\{\dfrac{\nu(g_k(t)) - \ell}{a-k}: 1\le k < a\right\}.
$$
After a base change, we may assume that $\mu\in \ZZ$. Clearly, $0 < \mu < m$. Let
$$
\widehat{h}(w,t) = \dfrac{g(t^{\mu}w,t)}{t^{a\mu +\ell}}.
$$
By our choice of $\mu$, $\widehat{h}(w,0)$ is a polynomial in $w$ of degree $a$ with at least two nonzero monomial terms. So $\widehat{h}(w,0)$
has a nonzero root. It follows that after a base change, there exists
$\widehat{\gamma}(t)\in \CC[[t]]$ such that $\widehat{\gamma}(0)\ne 0$ and $\widehat{h}(\widehat{\gamma}(t),t)\equiv 0$. So $g(\widehat{\beta}(t),t) \equiv 0$ for $\widehat{\beta}(t) = t^\mu \widehat{\gamma}(t)$.

Since $\Gamma_X$ is the only component of $\psi^{-1}(Q)$ that is flat over $\Delta$, we must have $\beta(t)\equiv \widehat{\beta}(t)$. This implies that $\delta = \mu$ and hence $h(w,t) \equiv \widehat{h}(w,t)$. But
$$
\deg h(w,0) = 2n < a = \deg \widehat{h}(w,0)
$$
which is contradiction. This proves \eqref{TRIGCFSLEM002E018} and hence \eqref{TRIGCFSLEM002E003}.

After we replace $f(x,y,t)$ by $\sigma(f(x,y,t)(\theta(y,t))^{-1})$, we obtain
\begin{equation}\label{TRIGCFSLEM002E007}
g(v,t) = v^{2n} \psi^* (f(x,y,t)) = (t^m - \alpha(t) v)^{2n}
\end{equation}
Let us rewrite $f(x,y,t)$ as
\begin{equation}\label{TRIGCFSLEM002E013}
\begin{aligned}
f(x,y,t) &= \left(x + \dfrac{t^{2c}}4\right)^n + b \left(x + \dfrac{t^{2c}}4\right)^{n-1}y\\
 &\quad + \sum_{p=0}^{n-1}\sum_{q=0}^\infty
\widehat{f}_{p,q}(t) \left(x + \dfrac{t^{2c}}4\right)^p y^q 
\\
\text{where } b &= b_1 + b_2 + ... + b_{n-1}\ne 0,\ \widehat{f}_{p,q}(t) \in \CC[[t]],\\
\widehat{f}_{p,q}(t) &\equiv 0 \hspace{6pt}\text{for all } p\ge 1 \text{ and } q\ge 2, \hspace{12pt}\text{and}\\
\widehat{f}_{p,q}(0) &= 0 \hspace{6pt}\text{for all }p,q.
\end{aligned}
\end{equation}
Then
$$
\begin{aligned}
g(v,t) 
&= \left(t^{2m} + t^{m+c} v + \dfrac{t^{2c}v^2}4\right)^{n} + b \left(t^{2m} + t^{m+c} v + \dfrac{t^{2c}v^2}4 \right)^{n-1} v^3\\
&\hspace{12pt} + \sum_{p=0}^{n-1}\sum_{q=0}^\infty
\widehat{f}_{p,q}(t) \left(t^{2m} + t^{m+c} v + \dfrac{t^{2c}v^2}4 \right)^{p} v^{2(n-p) + q}\\
&= \left(t^{m} + \dfrac{t^{c}v}2\right)^{2n} + b \left(t^{m} + \dfrac{t^{c}v}2\right)^{2n-2} v^3\\
&\hspace{12pt} + \sum_{p=0}^{n-1}\sum_{q=0}^\infty
\widehat{f}_{p,q}(t) \left(t^{m} + \dfrac{t^{c}v}2\right)^{2p} v^{2(n-p) + q}
\end{aligned}
$$
Thus
\begin{equation}\label{TRIGCFSLEM002E014}
\begin{aligned}
(t^m - \alpha(t) v)^{2n} &\equiv \left(t^{m} + \dfrac{t^{c}v}2\right)^{2n} + b \left(t^{m} + \dfrac{t^{c}v}2\right)^{2n-2} v^3\\
&\hspace{12pt} + \sum_{p=0}^{n-1}\sum_{q=0}^\infty
\widehat{f}_{p,q}(t) \left(t^{m} + \dfrac{t^{c}v}2\right)^{2p} v^{2(n-p) + q}
\end{aligned}
\end{equation}
Comparing the coefficients of $v$ on the two sides of \eqref{TRIGCFSLEM002E014} yields
\begin{equation}\label{TRIGCFSLEM002E016}
\alpha(t) = -\dfrac{t^{c}}2
\end{equation}
Therefore,
\begin{equation}\label{TRIGCFSLEM002E015}
b \left(t^{2m} + \dfrac{t^{c}v}2\right)^{2n-2} v^3
+ \sum_{p=0}^{n-1}\sum_{q=0}^\infty
\widehat{f}_{p,q}(t) \left(t^{m} + \dfrac{t^{c}v}2\right)^{2p} v^{2(n-p) + q} \equiv 0.
\end{equation}
It is easy to see that \eqref{TRIGCFSLEM002E015} is impossible since $b\ne 0$.
\end{proof}

\begin{proof}[Proof of Proposition \ref{TRIGCFSPROP003}]
Let $Z = \psi(X)$ be the proper transform of $X$ under $\psi$.
Suppose that there is such an irreducible component $M$. Then $Z_0$ contains the component $J$ with multiplicity $2$. By the flatness of $Q$ over $\Delta$, $Q_0^2 = Q_t^2$ and $Q_0 K_{R_0} = Q_t K_{R_t}$. Hence $Q_0\in |B+3F|$ on $R_0$. Similarly, by the flatness of $Z$ over $\Delta$, $Z_0\in |3B + 8F|$. So we must have
\begin{equation}\label{TRIGCFSPROP003E002}
Z_0 = B + 2J
\end{equation}
where $B$ is the $(-3)$-curve on $R_0$ and $J$ is a smooth rational curve in $|B+4F|$.

Let $\pi: \widehat{X}\to X$ be the minimal resolution of the rational map $\psi: X\dashrightarrow R$. That is, $\pi$ consists of a sequence of blowups at points such that the map $\psi \circ \pi$ is regular as in the diagram
$$
\begin{tikzcd}
\widehat{X}\ar{d}{\pi} \ar{rd}{\psi\circ \pi}\\
X \ar[dashed]{r}{\psi} & R
\end{tikzcd} 
$$
and it is minimal in the sense that for every $(-1)$-curve $E\subset \widehat{X}$ with $\pi_* E = 0$, $(\psi\circ \pi)_* E \ne 0$. Since $B$ is the only component of $Z_0$ which might not be dominated by a component of $X_0$, we see that
\begin{itemize}
\item there exists at most one $(-1)$-curve $E\subset \widehat{X}$ such that $\pi_* E = 0$;
\item if such $E$ exists, $(\psi\circ \pi)_* E = B$ and $\widehat{X}_0$ is reduced along $E$.
\end{itemize}
Note that if there is no such $E$, $\psi$ is regular to start with and $\widehat{X} = X$.

From the above discussion, the exceptional divisor $E_\pi\subset \widehat{X}$ of $\pi$ is either empty or a chain of smooth rational curves
\begin{equation}\label{TRIGCFSPROP003E000}
E_\pi = E_1\cup E_2\cup ...\cup E_{m-1}\cup E_m
\end{equation}
where $(\psi\circ \pi)_* E_m = B$, $E_m^2 = -1$, $(\psi\circ \pi)_* E_i = 0$, $E_i^2 = -2$, $E_{i} E_{i+1} = 1$ and $E_j E_k = 0$ for $1\le i \le m-1$ and $|j-k| \ge 2$.

Here by a chain of curves, we referred to a reduced curve whose dual graph is a chain.

Let $q = \Gamma\cap M$. Since $X$ is smooth and $\Gamma$ is a section of $X/\Delta$, $X_0$ is smooth at $q$. We claim

\begin{Claim}\label{TRIGCFSPROP003CLM000}
The rational map $\psi: X\dashrightarrow R$ is not regular at $q$.
\end{Claim}

Suppose that $\psi$ is regular at $q$. Let $\widehat{q} = \pi^{-1}(q)$ and $s = \psi(q)$. Since $\psi(\Gamma)\subset Q$, $s\in Q_0$.
Clearly, $J$ and $Q_0$ meet properly at $s$. Therefore,
$$
(\psi\circ \pi)^* Q = 8\widehat{\Gamma} + \Sigma
$$
where $\widehat{\Gamma}\subset \widehat{X}$ is the proper transform of $\Gamma$ under $\pi$ and $\Sigma$ is an effective divisor with $\supp(\Sigma)\subset \widehat{X}_0$, $\psi\circ\pi(\Sigma) \subset B$ and $\widehat{\Gamma}\cap \Sigma = \emptyset$.

Let $\widehat{M}\subset \widehat{X}$ be the proper transform of $M$ under $\pi$. Since $\widehat{M}\not\subset \Sigma$, $\Sigma^2 < 0$ if $\Sigma\ne 0$. On the other hand,
$$
\Sigma^2 = \Sigma(8\widehat{\Gamma} + \Sigma) = \Sigma . (\psi\circ \pi)^* Q = (\psi\circ \pi)_* \Sigma . Q = 0
$$
since $\psi\circ\pi(\Sigma) \subset B$ and $BQ = 0$. Therefore, $\Sigma = 0$, $B\not\subset Q_0$ and hence $Q_0$ is smooth. And since $(\psi\circ \pi)^* Q = 8\widehat{\Gamma}$, we see that this is impossible by Lemma \ref{TRIGCFSLEMDOUBLE}. Thus, we have proved Claim \ref{TRIGCFSPROP003CLM000}.

So $\psi$ is not regular at $q$. Consequently, $\widehat{\pi}: \widehat{X}\to X$ is a sequence of blowups over $q$ and $\pi(E_\pi) = q$.
Indeed, $\widehat{M} E_1 = 1$ and
$$
\widehat{M}\cup E_\pi = \widehat{M}\cup E_1\cup ...\cup E_{m-1}\cup E_m
$$
is a chain of curves. It is also clear that $\widehat{\Gamma} E_\pi = 1$. 

Let $\widehat{q} = \widehat{M}\cap E_{1}$. Since $(\psi\circ \pi)_* (E_1 + ... + E_m) = B$, $s = \psi\circ \pi(\widehat{q})\in B$. And since $s\in J\cap Q_0$, $B\cap Q_0 \ne\emptyset$ and hence $B\subset Q_0$ as $Q_0\in |B+3F|$.

We claim

\begin{Claim}\label{TRIGCFSPROP003CLM001}
The maps $\psi: M\to J$ and $\psi\circ \pi: \widehat{M}\to J$ are totally ramified at $q$ and $\widehat{q}$, respectively.
\end{Claim}

Otherwise, there exists a point $\widehat{r}\ne \widehat{q}$ on $\widehat{M}$ such that $\psi\circ \pi(\widehat{r}) = s$. Then
\begin{equation}\label{TRIGCFSPROP003E001}
(\psi\circ \pi)^* Q = 8\widehat{\Gamma} + \Sigma_1 + \Sigma_2
\end{equation}
where $\supp(\Sigma_1) = E_\pi$ and $\Sigma_2$ consists of all components disjoint from $E_\pi$. Since $\widehat{r}\in \Sigma_2$, $\Sigma_2\ne \emptyset$. So $\Sigma_2^2 < 0$. On the other hand,
$$
\Sigma_2^2 = \Sigma_2 (8\widehat{\Gamma} + \Sigma_1 + \Sigma_2) = \Sigma_2 . (\psi\circ \pi)^* Q = (\psi\circ \pi)_* \Sigma_2 . Q = 0.
$$
This proves Claim \ref{TRIGCFSPROP003CLM001}. A similar argument proves the following claim

\begin{Claim}\label{TRIGCFSPROP003CLM002}
$Z_0\cap Q_0 = B$.
\end{Claim}

Otherwise, there exists a point $r\in J\cap Q_0$ such that $r\ne s$. Let $\widehat{r}$ be a point on $\widehat{M}$ with $\psi\circ \pi(\widehat{r}) = r$. Again, we can put $(\psi\circ \pi)^* Q$ in the form of \eqref{TRIGCFSPROP003E001}. The same argument applies.

From Claim \ref{TRIGCFSPROP003CLM002}, we have
\begin{equation}\label{TRIGCFSPROP003E003}
Q_0 = B + 3G
\end{equation}
where $G$ is the curve in $|F|$ passing through $s$. Note that $G$ and $J$ meet transversely at $s$.

We can contract $E_1\cup E_2\cup ...\cup E_{m-1}$ in $\widehat{X}$ to obtain a surface $\overline{X}$ with an $A_{m-1}$ singularity. We have the diagram
\begin{equation}\label{TRIGCFSPROP003E004}
\begin{tikzcd}
\widehat{X}\ar[bend right=30]{dd}[left]{\pi} \ar{rrdd}{\psi\circ \pi}
\ar{d}{\varphi}\\
\overline{X} \ar{d}\ar{rrd}\\
X \ar[dashed]{rr}{\psi} && R
\end{tikzcd}
\end{equation}
where $\varphi: \widehat{X}\to \overline{X}$ is the contraction of $E_1\cup E_2\cup ...\cup E_{m-1}$. 

The map $\overline{X}\to R$ is regular and finite onto its image $Z$
since
$$(\psi\circ \pi)^{-1}(s) = E_1\cup E_2\cup ... \cup E_{m-1}.$$
So by Lemma \ref{TRIGCFSLEM000}, $Z$ is locally given by either
\begin{equation}\label{TRIGCFSPROP003E005}
Z \cong \big\{x y^2  = t^{2m}\big\} \subset \Delta_{xyt}^3 
\end{equation}
or
\begin{equation}\label{TRIGCFSPROP003E006}
Z\cong \big\{xy^2 = t^{2m} + t^{m+c} y\big\} \subset \Delta_{xyt}^3 \hspace{12pt} \text{for some } 0 < c < m
\end{equation}
at $s$.

From Claim \ref{TRIGCFSPROP003CLM002}, we also have
\begin{equation}\label{TRIGCFSPROP003E007}
Z.Q = d B + 8 \Gamma_Z
\end{equation}
on $R$ for some $d\in \ZZ^+$ and $\Gamma_Z = \psi\circ \pi(\widehat{\Gamma})$.

We now derive a contraction to \eqref{TRIGCFSPROP003E007}:  we can apply Lemma \ref{TRIGCFSLEM001} in the case of \eqref{TRIGCFSPROP003E005} and 
Lemma \ref{TRIGCFSLEM002} in the case of \eqref{TRIGCFSPROP003E006}.
The key is the verification of \eqref{TRIGCFSLEM001E002} and \eqref{TRIGCFSLEM002E002}, which follows from the fact that $G$ and $J$ meet transversely at $s$.

This finishes the proof of Proposition \ref{TRIGCFSPROP003}. 
\end{proof}

\section{Canonically Fibered Surfaces of Relative Genus $3$ and $4$}

\subsection{Some easy inequalities for canonically fibered surfaces}

In this section, we consider a canonically fibered surface of relative genus $3$ or $4$. We assume that $f: X\to C$ satisfies C1-C5 with $g(X_t) = 3$ or $4$ for $t\in C$ general. For such fibrations, as noted by Sun \cite[p.~ 161]{Sun1994}, in \cite{problem} Xiao asked for the value of
\begin{equation}\label{TRIGCFSE400}
\varliminf_{p_g(X)\to \infty} \dfrac{K_X^2}{p_g(X)}.
\end{equation}
For $g(X_t) = 3$, this number is between $24/5$ and $6$. For $g(X_t) = 4$, this number is at least $48/7$. Note that the general fibers of such fibrations are trigonal. For the case of genus $3$ or $4$ hyperelliptic canonical fibrations, X. Lu obtained very good bounds in \cite[Theorems 1.3, 1.5 and 1.6]{Lu:2020:Canonically:fibered:gen:type:surfaces}.

In what follows, we want to focus on the question of obtaining lower bounds for $K^2_X$ for the case of $g = 3,4$ nonhyperelliptic canonical fibrations. We start with some easy inequalities well known to the experts (cf. \cite{Sun1994}).

Suppose that
$$
N = \sum \mu_i \Gamma_i + V
$$
in \eqref{TRIGCFSE000}, where $\Gamma_i$ are the irreducible horizontal components of $N$ flat over $C$ with multiplicities $\mu_i > 0$ and where $V$ is the vertical part, so that in particular $f_* V = 0$.

We claim that
\begin{equation}\label{TRIGCFSE404}
K_X^2 \ge d \left(2g-2 + \sum \dfrac{\mu_i \gamma_i}{\mu_i+1}\right) + (\deg K_C) \sum \dfrac{\mu_i^2}{\mu_i+1}
\end{equation}
where $g = g(X_t)$ and $\gamma_i > 0$ is the degree of the map $\Gamma_i\to C$. This inequality should be well known to the experts.

Since $K_X$ is relatively nef, $K_X V \ge 0$ and hence
\begin{equation}\label{TRIGCFSE405}
\begin{aligned}
K_X^2 &= K_X(f^* D + \sum \mu_i \Gamma_i + V)\\
&\ge (2g-2) d + \sum \mu_i K_X \Gamma_i.
\end{aligned}
\end{equation}

For each $\Gamma_i$, we have
\begin{equation}\label{TRIGCFSE401}
K_X \Gamma_i + \Gamma_i^2 = 2p_a(\Gamma_i) - 2 \ge \deg K_C .
\end{equation}
On the other hand,
\begin{equation}\label{TRIGCFSE402}
K_X\Gamma_i = (f^* D + \sum \mu_j \Gamma_j + V) \Gamma_i
\ge d \gamma_i + \mu_i \Gamma_i^2.
\end{equation}
Combining \eqref{TRIGCFSE401} and \eqref{TRIGCFSE402}, we obtain
\begin{equation}\label{TRIGCFSE403}
K_X \Gamma_i \ge \dfrac{d\gamma_i}{\mu_i+1} + \left(\dfrac{\mu_i}{\mu_i+1}\right)\deg K_C \text{.}
\end{equation}
Together with \eqref{TRIGCFSE405}, we obtain \eqref{TRIGCFSE404}. Thus,
\begin{equation}\label{TRIGCFSE406}
\varliminf_{p_g(X)\to \infty} \dfrac{K_X^2}{p_g(X)} \ge 2g-2 + \sum \dfrac{\mu_i \gamma_i}{\mu_i+1}.
\end{equation}
Note that
$$
\sum \mu_i \gamma_i = 2g - 2.
$$

Obviously, the minimum of the right hand side (RHS) of \eqref{TRIGCFSE406} is achieved when the horizontal part of $N$ has one irreducible component with multiplicity $\mu_1 = 2g - 2$, i.e.,
\begin{equation}\label{TRIGCFSE407}
K_X = f^*D + (2g-2) \Gamma + V
\end{equation}
for a section $\Gamma$ of $X/C$.

\subsection{Canonically fibered surfaces of relative genus $3$}

Here we will prove Theorem \ref{TRIGCFSTHMG36} when $g=3$.
We just have to deal with the cases
\begin{equation}\label{TRIGCFSTHMG36E000}
K_X = f^*D + 4 \Gamma + V
\end{equation}
and
\begin{equation}\label{TRIGCFSTHMG36E001}
K_X = f^*D + \Gamma_1 + 3\Gamma_2 + V.
\end{equation}
For all other configurations of $N$, we already have
\begin{equation}\label{TRIGCFSTHMG36E002}
K_X^2 \ge \dfrac{16}3 d + \dfrac{8}3 \deg K_C
\end{equation}
by \eqref{TRIGCFSE404}.

Let us first treat the case \eqref{TRIGCFSTHMG36E000}.
We have the following statement, similar to Theorem \ref{TRIGCFSTHMPCM}.

\begin{Proposition}\label{TRIGCFSPROP001}
Let $f: X\to C$ be a canonically fibered surface satisfying C1-C5 and \eqref{TRIGCFSTHMG36E000} with $\Gamma$ a section of $f$
and let $\tau$ be the rational map 
$$
\begin{tikzcd}
X\ar[dashed]{r}{\tau} & W = \PP (f_* \OO_X(4\Gamma))^\vee.
\end{tikzcd}
$$
If $X_p$ is not hyperelliptic for $p\in C$ general, then
\begin{itemize}
\item $\tau$ is regular;
\item $\tau: X_p\to W_p$ is a closed embedding for $p\in C$ general; 
\item $Y_p$ is an integral quartic curve in $W_p\cong \PP^2$ for $Y = \tau(X)$ and all $p\in C$.
\end{itemize}
\end{Proposition}

\begin{proof}
Since $f$ is not hyperelliptic, $\tau$ is a closed embedding over $p\in C$ general.

For each $p\in C$, let $B_p$ be the unique irreducible component of $X_p$ such that $B_p\cap \Gamma \ne \emptyset$.
Then $J = \tau(B_p)$ is a linearly non-degenerate curve in $W_p\cong \PP^2$. And since $\deg Y_p \le 4$, $\tau$ maps $B_p$ either birationally or $2$-to-$1$ onto $J$.

If $\tau$ maps $B_p$ birationally onto $J$, then $p_a(J)\ge 3$ by
Lemma \ref{TRIGCFSLEMGENUS}. And since $\deg J\le 4$, we conclude that $J$ is an integral quartic curve on $W_p$. Hence $J$ is the only component of $Y_p$, $Y_p$ is an integral quartic curve and $\tau$ is regular along $X_p$.

Suppose that $\tau$ maps $B_p$ $2$-to-$1$ onto $J$. Then $J$ is a smooth conic curve on $W_p$. Since $\deg Y_p\le 4$ and $Y_p$ contains $J$ with multiplicity $\ge 2$, we conclude that $J$ is the only component of $Y_p$ and $\tau$ is regular along $X_p$. In conclusion, $\tau$ is regular. 

In addition, by Lemma \ref{TRIGCFSLEMDOUBLE}, the existence of $\Lambda\in |\OO_W(1)|$ with $\tau^* \Lambda = 4\Gamma$ implies that $\tau: B_p\to J$ cannot be $2$-to-$1$. So $Y_p$ is an integral quartic curve for all $p\in C$. 
\end{proof}

\begin{proof}[Proof of Theorem \ref{TRIGCFSTHMG36} in the case \eqref{TRIGCFSTHMG36E000}]
By Proposition \ref{TRIGCFSPROP001}, $Y$ is smooth in codimension one. And since $Y$ is a divisor of $W$, it is Cohen-Macaulay.
Then Serre's criterion tells us that $Y$ is normal. And since $\tau: X\to Y$ is finite along $\Gamma$, it is a local isomorphism along $\Gamma$.
It follows that $Y_p$ is smooth at $Y_p\cap \Lambda$ for all $p\in C$, where $\Lambda$ is the unique section in $|\OO_W(1)| \cong |\OO_X(4\Gamma)|$.

Let us consider the rational map
$$
\begin{tikzcd}
X\ar[dashed]{r}{\alpha} & S = \PP (f_* \OO_X(3\Gamma))^\vee.
\end{tikzcd}
$$
Clearly, $\alpha$ factors through $\tau$ with the diagram
$$
\begin{tikzcd}
X\ar[dashed]{dr}[below]{\alpha} \ar{r}{\tau} & W = \PP (f_* \OO_X(4\Gamma))^\vee \ar[dashed]{d}{\beta}\\ 
& S = \PP (f_* \OO_X(3\Gamma))^\vee.
\end{tikzcd}
$$
Geometrically, $\beta$ is the linear projection from the point $Y_p\cap \Lambda$ over every point $p\in C$. And since $Y_p$ is smooth at the point $Y_p\cap \Lambda$ for all $p\in C$, $\beta$ is regular along $Y$. It follows that $\alpha$ is regular.

Since $S$ is a ruled surface, $\tau^* \Gamma_S = 3\Gamma$ for $\Gamma_S = \tau(\Gamma)$ and $\deg \tau = 3$,
\begin{equation}\label{TRIGCFSTHMG36E004}
\Gamma^2 = \dfrac{1}3 \Gamma_S^2 \le \dfrac{1}{3}\left(-d + \dfrac{\deg K_C}2\right).
\end{equation}
Therefore,
\begin{equation}\label{TRIGCFSTHMG36E003}
\begin{aligned}
K_X^2 &= K_X(f^* D + 4\Gamma + V) \ge K_X(f^* D + 4\Gamma)\\
&= 4d + 4K_X \Gamma = 4d + 4\deg K_C - 4\Gamma^2\\
&\ge 4d + 4\deg K_C - \dfrac{4}{3}\left(-d + \dfrac{\deg K_C}2\right)\\
&= \dfrac{16}3 d + \dfrac{10}3 \deg K_C.
\end{aligned}
\end{equation}
\end{proof}

Next, let us treat the case \eqref{TRIGCFSTHMG36E001}.

\begin{Proposition}\label{TRIGCFSPROP004}
Let $f: X\to C$ be a canonically fibered surface satisfying C1-C5 
without the hypothesis that $K_X$ is relatively nef over $C$.
Suppose that \eqref{TRIGCFSTHMG36E001} holds with $\Gamma_1$ 
and $\Gamma_2$ two disjoint sections of $f$. If $X_p$ is not hyperelliptic for $p\in C$ general, then the rational map
$$
\begin{tikzcd}
X\ar[dashed]{r}{\alpha} & S = \PP (f_* \OO_{X}(3 \Gamma_2))^\vee.
\end{tikzcd}
$$
is regular.
\end{Proposition}

\begin{proof}
Let $\tau$ be the rational map 
$$
\begin{tikzcd}
X\ar[dashed]{r}{\tau} & W = \PP (f_* \OO_X(\Gamma_1 + 3\Gamma_2))^\vee
\end{tikzcd}
$$
and let $Y = \tau(X)$ be the proper transform of $X$.
Clearly, $\alpha$ factors through $\tau$ with the diagram
$$
\begin{tikzcd}
X\ar[dashed]{dr}[below]{\alpha} \ar[dashed]{r}{\tau} & W = \PP (f_* \OO_{X}(\Gamma_1 + 3\Gamma_2))^\vee \ar[dashed]{d}{\beta}\\
& S = \PP (f_* \OO_{X}(3\Gamma_2))^\vee.
\end{tikzcd}
$$
Geometrically, $\beta$ is the linear projection
$$
\begin{tikzcd}
W_p \cong \PP^2 \ar[dashed]{r} & S_p \cong \PP^1
\end{tikzcd}
$$
from the point $Y_p\cap \tau(\Gamma_1)$ over every point $p\in C$. 

Since $X_p$ is not hyperelliptic for $p\in C$ general and $\tau: X_p\to W_p$ is the canonical map of $X_p$, it is a closed embedding.
So $Y_p$ is a smooth quartic curve in $W_p\cong \PP^2$. In particular, it is smooth
at the point $Y_p\cap \tau(\Gamma_1)$. So $h^0(\OO_{X_p}(3\Gamma_2)) = 2$ and $S$ is a ruled surface over $C$.

Let $A_p$ and $B_p$ be the components of $X_p$ such that $A_p\cap \Gamma_1\ne\emptyset$ and $B_p\cap \Gamma_2\ne\emptyset$, respectively. If $A_p = B_p$, then $J = \tau(B_p)$ is a linearly non-degenerate curve in $W_p\cong \PP^2$. So $\deg J \ge 2$ and $\tau$ maps $B_p$ either birationally or $2$-to-$1$ onto $J$.

If $\tau$ maps birationally onto $J$, then $p_a(J) \ge 3$ by Lemma \ref{TRIGCFSLEMGENUS}. Hence $\deg J = 4$ and $Y_p = J$. So $\tau$ is regular along $X_p$ and $Y_p$ is reduced. By Serre's criterion, $Y$ is normal along $Y_p$. Hence
$$\tau^{-1}(Y_p\cap \tau(\Gamma_1)) = X_p\cap \Gamma_1$$
since $\tau$ has connected fibers over $Y_p$. So $\tau$ is finite over $Y_p\cap \tau(\Gamma_1)$ and hence a local isomorphism. Consequently, $Y_p$ is smooth at $Y_p\cap \tau(\Gamma_1)$ and the linear projection $\beta$ is regular along $Y_p$. The regularity of $\alpha$ along $X_p$ then follows.

Suppose that $\tau$ maps $2$-to-$1$ onto $J$. Then $J$ is a smooth conic curve and $Y_p = 2J$. So $\tau$ is regular along $X_p$. Let $\Lambda\in |\OO_W(1)| \cong |\OO_X(\Gamma_1 + 3\Gamma_2)|$ be the unique section. Since $\Lambda$ and $J$ meet properly, we have
\begin{equation}\label{TRIGCFSTHMG36E011}
\tau^* \Lambda = \Gamma_1 + 3\Gamma_2
\end{equation}
in an open neighborhood of $X_p$. Therefore,
\begin{equation}\label{TRIGCFSTHMG36E012}
\tau^*\Lambda . X_p = q_1 + 3 q_2
\end{equation}
for $q_1 = X_p\cap \Gamma_1$ and $q_2 = X_p\cap \Gamma_2$. Since $\Gamma_1$ and $\Gamma_2$ are disjoint, $q_1\ne q_2$.

If $\Lambda$ meets $J$ at two distinct points, either $\tau^{-1}(\Lambda) \cap X_p$ consists of more than two points or $\tau^* \Lambda . X_p = 2q_1 + 2q_2$. Otherwise, $\Lambda$ is tangent to $J$ at a unique point. Then either $\tau^{-1}(\Lambda) \cap X_p$ consists of a unique point or $\tau^* \Lambda . X_p = 2q_1 + 2q_2$. In either case, \eqref{TRIGCFSTHMG36E012} cannot hold. Therefore, $\tau$ cannot map $B_p$ $2$-to-$1$ onto $J$. This proves the proposition when $A_p = B_p$.

Suppose that $A_p \ne B_p$. If $\tau$ is regular at $q_1 = A_p\cap \Gamma_1$, then $I = \tau(A_p)$ is a line on $W_p\cong \PP^2$. Let $J = \tau(B_p)$ be the proper transform of $B_p$. Since $I\cup J$ is linearly non-degenerate, $J \ne \emptyset, I$. Let $M$ be the scheme-theoretic proper transform of $X_p$ under $\tau$, as defined in Lemma \ref{TRIGCFSLEMGENUS}. Then $p_a(M)\ge 3$. Hence $\deg M = 4$, $Y_p = M$ and $\tau$ is regular. If $Y_p$ is smooth at $\tau(q_1)$, then the linear projection $\beta$ is regular along $Y_p$ and we are done. Otherwise, $\tau(q_1) = I\cap J$. Then $\tau^{-1}(\tau(q_1))$ contains at least two connected components: one is $q_1$ and the other is a union $\Sigma$ of components of $X_p$ such that $\tau(\Sigma) = \tau(q_1)$, $\Sigma\cap A_p\ne \emptyset$ and $\Sigma\cap B_p \ne \emptyset$.
The map $\tau: A_p \to I$ is given by a linear system of $\OO_{A_p}(q_1)$. So $A_p\cong \PP^1$ and $\tau: A_p\to I$ is an isomorphism. But $\tau$ maps two distinct points $q_1$ and $\Sigma\cap A_p$ to the same point, which is a contradiction. This proves the proposition when $A_p\ne B_p$ and $\tau$ is regular at $q_1$.

Suppose that $\tau$ is not regular at $q_1$. Then $\tau$ contracts $A_p$ and $J = \tau(B_p)$ is linearly non-degenerate. So $2\le \deg J \le 3$ and $J$ is either a conic or a cubic.

Suppose that $J$ is a conic. Then $\tau$ is not regular at $q_2 = B_p\cap \Gamma_2$. The map $\tau: B_p\to J$ is given by a linear system of $\OO_{B_p}(2q_2)$. Hence $B_p\cong \PP^1$ and $\tau: B_p\to J$ is an isomorphism. Let
\begin{equation}\label{TRIGCFSTHMG36E013}
\begin{tikzcd}
\widehat{X} \ar{rd}{\widehat{\tau}} \ar{d}{g}\\
X \ar[dashed]{r}{\tau} & W
\end{tikzcd}
\end{equation}
be the minimal resolution of the rational map $\tau$. Let $E_i = g^{-1}(q_i)$ for $i=1,2$. Since $\widehat{\tau}_* E_i \ne 0$ and $\deg Y_p = 4$, $Y_p = J + L_1 + L_2$ for two lines $L_i = \widehat{\tau}_* E_i$. Clearly, $\widehat{\tau}^{-1}(L_1\cup L_2)$ consists of two connected components: one is a union $\Sigma$ of components of $\widehat{X}_p$ such that $\Sigma\supset E_1$ and the other is $E_2$. If $L_1\ne L_2$, $Y_p$ is reduced and hence $Y$ is normal along $Y_p$. Then $\widehat{\tau}$ has connected fibers over $Y_p$. But $\widehat{\tau}^{-1}(L_1\cup L_2)$ is not connected, which is a contradiction. So $L_1 = L_2$.
 
Let $\widehat{B}_p\subset \widehat{X}$ be the proper transform of $B_p$ under $g$. Since $\Sigma\cap E_1 = \emptyset$, $s_1 = \Sigma\cap \widehat{B}_p$ and $s_2 = E_2\cap \widehat{B}_p$ are two distinct points. And since $\widehat{\tau}:\widehat{B}_p\to J$ is an isomorphism, $s_1$ and $s_2$ map to two distinct points by $\widehat{\tau}$. So $L_i$ and $J$ meet at two distinct points $\widehat{\tau}(s_i)$ for $i=1,2$. Let $\nu: \widehat{Y} \to W$ be the normalization of $Y$. We have the diagram
$$
\begin{tikzcd}
\widehat{X} \ar{rd}{\widehat{\tau}} \ar{d}{g} \ar{r}{\phi} & \widehat{Y} \ar{d}{\nu}\\
X \ar[dashed]{r}{\tau} & W
\end{tikzcd}
$$
Let $V\subset W$ be an analytic open neighborhood of $\widehat{\tau}(s_1)$. Since $\widehat{\tau}^{-1}(\widehat{\tau}(s_1))$ has two connected components, $\widehat{\tau}^{-1}(V)$ has two connected components, correspondingly. And since $\phi: \widehat{X}\to \widehat{Y}$ has connected fibers, $\nu^{-1}(V) = U_1\sqcup U_2$ also has two connected components with
$$
U_1 \cap \widehat{Y}_p \cong V\cap (L_1\cup J) \hspace{12pt}\text{and}\hspace{12pt}
U_2 \cap \widehat{Y}_p \cong V\cap L_2
$$
In particular, $\widehat{Y}_p$ has normal crossings in $\nu^{-1}(V)$. The same holds over the point $\widehat{\tau}(s_2)$ and also over all points of $Y_p$. In conclusion, $\widehat{Y}_p$ has normal crossings and
$$
\widehat{Y}_p = \widehat{L}_1 \cup \widehat{J} \cup \widehat{L}_2
$$
where $\nu$ maps $\widehat{L}_1$, $\widehat{J}$ and $\widehat{L}_2$ isomorphically onto $L_1, J$ and $L_2$, respectively, $\widehat{L}_1 \cap \widehat{L}_2 = \emptyset$, $\widehat{L}_1 \cap \widehat{J} = \{1\text{ point}\}$ and $\widehat{L}_2 \cap \widehat{J} = \{1\text{ point}\}$. Then $p_a(\widehat{Y}_p) = 0$. But $p_a(\widehat{Y}_t) = 3$ for $t\in C$ general, which is a contradiction. This proves the proposition when $A_p \ne B_p$ and $\tau$ is regular at neither of $q_1$ and $q_2$.

Suppose that $J$ is a cubic. Then $\tau$ is regular at $q_2 = B_p\cap \Gamma_2$. Let
$g: \widehat{X}\to X$ be the minimal resolution of $\tau$ given by the diagram \ref{TRIGCFSTHMG36E013}. Then $Y_p = L_1 + J$ for a line $L_1 = \widehat{\tau}_* E_1$. Since $Y_p$ is reduced, $Y$ is normal along $Y_p$. So $\widehat{\tau}$ has connected fibers. In particular, $\widehat{\tau}^{-1}(\tau(q_2))$ contains a unique point. Hence
$\tau(q_2') \ne \tau(q_2)$ for all $q_2' \ne q_2\in B_p$. So $\Lambda . J = 3 \tau(q_2)$.

Also $\tau(q_2)\not\in L_1$; otherwise, $\widehat{\tau}^{-1}(\tau(q_2))$ will have at least two connected components. So the two lines $L_1$ and $\Lambda_p$ are distinct and
the intersection $L_1\cap \Lambda = Y_p\cap \tau(\Gamma_1)$ does not lie on the cubic curve $J$. Therefore, the finite map $\alpha: B_p \to S_p\cong \PP^1$ has degree $3$ as it is given by
the diagram
$$
\begin{tikzcd}
B_p \ar{r}{\tau} \ar{rd}{\alpha} & J\ar{d}{\beta}\\
& S_p
\end{tikzcd}
$$
where $\beta$ is the linear projection $W_p\dasharrow S_p$ from a point not lying on $J$. Consequently, $\alpha$ is regular on $X_p$. 
\end{proof}

\begin{proof}[Proof of Theorem \ref{TRIGCFSTHMG36} in the case \eqref{TRIGCFSTHMG36E001}]
Suppose that $\Gamma_1 \Gamma_2 = \delta$. Then there exists a birational morphism $\pi: \widehat{X}\to X$, consisting of $\delta$ blowups at points, such that the proper transforms $\widehat{\Gamma}_i$ of $\Gamma_i$ are disjoint for $i=1,2$. Note that
\begin{equation}\label{TRIGCFSTHMG36E005}
\Gamma_i^2 = \widehat{\Gamma}_i^2 + \delta
\end{equation}
for $i=1,2$.

By Proposition \ref{TRIGCFSPROP004}, the rational map
$$
\begin{tikzcd}
\widehat{X}\ar[dashed]{r}{\alpha} & \widehat{S} = \PP (f_* \OO_{\widehat{X}}(3\widehat{\Gamma}_2))^\vee.
\end{tikzcd}
$$
is regular. From the regularity of $\alpha$, we obtain
$$
\widehat{\Gamma}_2^2 \le \dfrac{1}{3}\left(-d + \dfrac{\deg K_C}2\right)
$$
similar to \eqref{TRIGCFSTHMG36E004}. Together with \eqref{TRIGCFSTHMG36E005}, we have
\begin{equation}\label{TRIGCFSTHMG36E006}
\Gamma_2^2 \le \dfrac{1}{3}\left(-d + \dfrac{\deg K_C}2\right) + \delta.
\end{equation}

For $K_X\Gamma_1$, we have
$$
\begin{aligned}
K_X\Gamma_1 &= (f^* D + \Gamma_1 + 3\Gamma_2 + V)\Gamma_1\\
&\ge (f^* D + \Gamma_1 + 3\Gamma_2)\Gamma_1 = d + \Gamma_1^2 + 3\delta\\
&= d + (\deg K_C - K_X\Gamma_1) +3\delta.
\end{aligned}
$$
Therefore,
\begin{equation}\label{TRIGCFSTHMG36E007}
K_X\Gamma_1 \ge \dfrac{d}2 + \dfrac{1}2\deg K_C + \dfrac{3}2\delta.
\end{equation}

For $K_X\Gamma_2$, we have
$$
\begin{aligned}
K_X\Gamma_2 &= (f^* D + \Gamma_1 + 3\Gamma_2 + V)\Gamma_2\\
&\ge (f^* D + \Gamma_1 + 3\Gamma_2)\Gamma_2 = d + \delta + 3\Gamma_2^2\\
&= d + \delta + 3(\deg K_C - K_X\Gamma_2).
\end{aligned}
$$
Therefore,
\begin{equation}\label{TRIGCFSTHMG36E008}
K_X\Gamma_2 \ge \dfrac{d}4 + \dfrac{3}4\deg K_C + \dfrac{\delta}4.
\end{equation}
Also from \eqref{TRIGCFSTHMG36E006}
\begin{equation}\label{TRIGCFSTHMG36E009}
K_X\Gamma_2 = \deg K_C - \Gamma_2^2 \ge \dfrac{d}3 + \dfrac{5}{6} \deg K_C - \delta.
\end{equation}

Combing \eqref{TRIGCFSTHMG36E007}-\eqref{TRIGCFSTHMG36E009}, we obtain
\begin{equation}\label{TRIGCFSTHMG36E010}
\begin{aligned}
K_X^2 &= K_X(f^* D + \Gamma_1 + 3\Gamma_2 + V)\\
&\ge K_X(f^* D + \Gamma_1 + 3\Gamma_2) = 4d + K_X\Gamma_1 + 3K_X\Gamma_2\\
&\ge 4d + \left(\dfrac{d}2 + \dfrac{1}2\deg K_C + \dfrac{3}2\delta\right) + \dfrac{6}5\left(\dfrac{d}4 + \dfrac{3}4\deg K_C + \dfrac{\delta}4\right)\\
&\hspace{24pt} + \dfrac{9}5 \left(\dfrac{d}3 + \dfrac{5}{6} \deg K_C - \delta\right)\\
&= \dfrac{27}5 d + \dfrac{29}{10} \deg K_C.
\end{aligned}
\end{equation}
\end{proof}

So we have three lower bounds for $K_X^2$:
In the case \eqref{TRIGCFSTHMG36E000}, we have \eqref{TRIGCFSTHMG36E003}; in the case \eqref{TRIGCFSTHMG36E001}, we have \eqref{TRIGCFSTHMG36E010}; otherwise, we have
\eqref{TRIGCFSTHMG36E002}. In conclusion, we obtain \eqref{TRIGCFSE408} for $g=3$.

\subsection{Canonically fibered surfaces of relative genus $4$}

Here we will prove Theorem \ref{TRIGCFSTHMG36} when $g=4$.
We just have to treat the case
\begin{equation}\label{TRIGCFSTHMG36E020}
K_X = f^*D + 6 \Gamma + V.
\end{equation}
For all other configurations of $N$, we already have
\begin{equation}\label{TRIGCFSTHMG36E021}
K_X^2 \ge \dfrac{22}3 d + \dfrac{14}3 \deg K_C
\end{equation}
by \eqref{TRIGCFSE404}.

\begin{Proposition}\label{TRIGCFSPROP002}
Let $f: X\to C$ be a canonically fibered surface satisfying C1-C5 and \eqref{TRIGCFSTHMG36E020} with $\Gamma$ a section of $f$
and let $\tau$ be the rational map 
$$
\begin{tikzcd}
X\ar[dashed]{r}{\tau} & W = \PP (f_* \OO_X(6\Gamma))^\vee.
\end{tikzcd}
$$
If $X_p$ is not hyperelliptic for $p\in C$ general, then
\begin{itemize}
\item $\tau$ is regular;
\item $\tau: X_p \to W_p$ is a closed embedding for $p\in C$ general;
\item $Y = \tau(X)\subset Q$ for a hypersurface $Q\subset W$ such that $Q_p$ is an integral quadratic surface in $W_p\cong \PP^3$ for all $p\in C$;
\item $Y_p$ is an integral curve in $|\OO_W(3)\otimes \OO_{Q_p}|$ for all $p\in C$.
\end{itemize}
\end{Proposition}

\begin{proof}
For $p\in C$ general, since $X_p$ is not hyperelliptic, the map
$$
\begin{tikzcd}
\operatorname{Sym}^2 H^0(K_{X_p}) \ar{r} & H^0(2K_{X_p})
\end{tikzcd}
$$
is surjective and hence its kernel has dimension
$$
\dim \operatorname{Sym}^2 H^0(K_{X_p}) - \dim H^0(2K_{X_p}) = 10 - 9 = 1.
$$
And since $Y_p = \tau(X_p)$ is the canonical embedding of $X_p$, there exists a unique quadratic surface $Q_p\subset W_p$ containing $Y_p$. We let
$$
Q = \overline{\bigcup_{X_p \text{ smooth}} Q_p}.
$$
Clearly, $Q$ is a flat family of quadrics over $C$ containing $Y$.

For each $p\in C$, let $B_p$ be the unique irreducible component of $X_p$ such that $B_p\cap \Gamma \ne \emptyset$.
Then $J = \tau(B_p)\subset Y_p\subset Q_p$ is a linearly non-degenerate curve in $W_p\cong \PP^3$. So the quadric $Q_p$ must be integral. That is, it is either a smooth quadric or a cone over a smooth conic curve.

Let $\LL = \OO_X(6\Gamma)$ and let us consider the map
\begin{equation}\label{TRIGCFSE824}
\begin{tikzcd}
\operatorname{Sym}^3 f_* \LL \ar{r} & f_* \LL^{\otimes 3}
\end{tikzcd}
\end{equation}
At a general point $p\in C$, the above map is
$$
\begin{tikzcd}
\operatorname{Sym}^3 f_* \LL \Big|_p \ar{r} \ar[equal]{d}& f_* \LL^{\otimes 3} \Big|_p \ar[equal]{d}\\
\operatorname{Sym}^3 H^0(K_{X_p}) \ar{r} & H^0(3K_{X_p})
\end{tikzcd}
$$
and it is surjective with kernel of dimension
$$
\dim \operatorname{Sym}^3 H^0(K_{X_p}) - \dim H^0(3K_{X_p}) = 20 - 15 = 5.
$$
Then the kernel of \eqref{TRIGCFSE824}, denoted by $\VV$, is a vector bundle of rank $5$ on $C$. At a general point $p\in C$, we can identify
$$
\VV_p \cong H^0(I_{Y_p}(3))
$$
where $I_{Y_p}$ is the ideal sheaf of $Y_p$ in $W_p$. For arbitrary $p\in C$, we have an inclusion
$$
\VV_p \subset H^0(I_{Y_p}(3))
$$
That is, $\PP \VV_p$ is a sub-linear system of cubic surfaces containing $Y_p$. And since $\dim \VV_p = 5$, there exists a cubic surface $R_p\in \PP \VV_p$ such that $Q_p \not \subset R_p$.

Over an affine open set $U\subset C$, we can find $R\in \PP H^0(U, \VV)$ such that $Q_p \not\subset R_p$ for all $p\in U$. We have $Y\subset Q\cap R$ over $U$. And since $Y_p$ is a smooth curve of degree $6$ in $W_p$ for $p\in U$ general, we must have $Y_p = Q_p\cap R_p$ for $p\in U$ general. That is, $Y = Q\cap R$ over $U$. This proves that $Y$ is a local complete intersection and $Y_p \in |\OO_{Q_p}(3)|$ for all $p\in C$.

For each $p\in C$, by the same argument as before,
$\tau$ maps $B_p$ either birationally or $2$-to-$1$ onto $J$.

Suppose that $\tau$ maps $B_p$ birationally onto $J$. Then $p_a(J)\ge 4$ by Lemma \ref{TRIGCFSLEMGENUS}. 

If $Q_p$ is smooth, then $Q_p \cong \PP^1\times \PP^1$. Suppose that $J$ has a curve of type $(a,b)$. Then $p_a(J) = (a-1)(b-1)\ge 4$. And since $J\subset Q_p\cap R_p$, $a,b\le 3$. We must have $a=b=3$. Hence
$J = Y_p$ and $\tau$ is regular along $X_p$.

If $Q_p$ is not smooth, then $Q_p$ is a cone over a smooth conic curve since it is integral. Let $\pi: \widehat{Q}_p\cong \FF_2\to Q_p$ be the minimal resolution of $Q_p$ and let $\widehat{J}\subset \widehat{Q}_p$ be the proper transform of $J$. Suppose that $\widehat{J}\in |aB + bF|$, where $B$ and $F$ are effective generators of $\Pic(\widehat{Q}_p)$ satisfying that $B^2=-2$, $BF=1$ and $F^2 = 0$. Then the Hilbert polynomial of $J$ is
$$
\begin{aligned}
\chi(\OO_J(n)) &= h^0(\OO_{Q_p}(n)) - h^0(\OO_{Q_p}(n) \otimes I_J)\\
&= h^0(\OO_{\widehat{Q}_p}(nB + 2nF)) - h^0(\OO_{\widehat{Q}_p}((n-a)B + (2n-b) F))\\
&= bn - \left(b - \left\lceil \dfrac{b}2 \right\rceil-1\right)
\left(\left\lceil \dfrac{b}2 \right\rceil - 1\right) + 1
\end{aligned}
$$
for $n\gg0$, where $I_J$ is the ideal sheaf of $J$ in $Q_p$. So we have
$$
p_a(J) = \left(b - \left\lceil \dfrac{b}2 \right\rceil-1\right)
\left(\left\lceil \dfrac{b}2 \right\rceil - 1\right)
$$
and hence $b\ge 6$. And since $\deg J = b \le \deg Y_p = 6$, we conclude that $b=6$. Hence $\deg J = \deg Y_p$ and $p_a(J) = p_a(Y_p)$. 
It follows that $J = Y_p$ and $\tau$ is regular along $X_p$.

Suppose that $\tau$ maps $B_p$ $2$-to-$1$ onto $J$. Then $J$ is a cubic rational normal curve in $W_p$. Since $Y_p = Q_p\cap R_p$ is a complete intersection of degree $6$ and it contains $J$ with multiplicity $\ge 2$, we conclude that $J$ is the only component of $Y_p$ and $\tau$ is regular along $X_p$. In conclusion, $\tau$ is regular. 

We claim that this case $Y_p = 2J$ cannot occur, similar to the situation of $g = 5$.

If $Q_p$ is smooth, then $J$ is a curve on $Q_p\cong \PP^1\times \PP^1$ of type $(1,2)$ or $(2,1)$. On the other hand, $Y_p = 2J = Q_p\cap R_p$ is a curve on $Q_p$ of type $(3,3)$. This is obviously impossible.

Suppose that $Q_p$ is singular. Then it is a cone over a smooth conic curve. Let $o$ be its vertex. Clearly, $o\in J$.

Let $\Lambda\in |\OO_W(1)|\cong |\OO_X(6\Gamma)|$ be the unique section. Then $\tau^* \Lambda = 6\Gamma$. So $\Lambda$ meets $J$ at a unique point $q$ such that $\tau: B_p\to J$ is totally ramified over $q$.

If $\Lambda$ does not pass through $o$, then $q\ne o$ and $\Lambda$ is smooth at $q$.
This is impossible by Lemma \ref{TRIGCFSLEMDOUBLE}. So $q=o\in \Lambda$.

Let us ``resolve'' the vertex $o$ of $Q_p$ using the Fano variety of lines on $Q$.
In the same way as we resolve the singularities of a family of rational normal scrolls in Lemma \ref{TRIGCFSLEMCUBICFIB}, we let
$$
S = \big\{(p, [L]): p\in C, [L]\in {\mathbb{G}\text{r}}(W_p, 1), L \subset Q_p\big\}
$$
be the relative Fano variety of lines on $Q$ over $C$ and let
$$
V = \{(p, [L], s): (p, [L])\in S, s\in L\}\subset S\times_C W
$$
be the universal family over $S$.

If all fibers $Q_p$ of $Q/C$ are singular, then $S_p\cong \PP^1$ and $V_p\cong \FF_2$ for all $p\in C$. We have the diagram
\begin{equation}\label{TRIGCFSTHMG36E017}
\begin{tikzcd}
& V\ar{d}{\xi} \ar{r}{\phi} & S \ar{d}{\pi}\\
X \ar[dashed]{ur}{\psi}\ar{r}{\tau} & W \ar{r} & C
\end{tikzcd} 
\end{equation}
where we let $Z = \psi(X)$ be the proper transform of $X$ under $\psi$ and $D = \xi^* \Lambda$.

For $p\in C$ general, $Y_p = Q_p \cap R_p$ is smooth and hence $Y_p$ does not pass through the vertex of $Q_p$. So $Z_p \in |3B + 6F|$, where $B$ and $F$ are effective generators of $\Pic(Q_p) = \Pic(\FF_2)$ satisfying $B^2 = -2$, $BF=1$ and $F^2 = 0$.

For $p\in C$ such that $Y_p = 2J$, we have $\widehat{J} \in |B + 3F|$ for the proper transform 
$\widehat{J}\subset Z$ of $J$ and
\begin{equation}\label{TRIGCFSTHMG36E014}
Z_p = B + 2\widehat{J}.
\end{equation}
Since $\Lambda$ passes through the vertex of $Q_p$, $B\subset D_p$. We must have
\begin{equation}\label{TRIGCFSTHMG36E015}
D_p = B + 2G
\end{equation}
where $G\in |F|$ is the curve passing through the point $B\cap \widehat{J}$. Otherwise, if $D_p$ contains a component $G'\in |F|$ such that $B\cap \widehat{J}\cap G' = \emptyset$, then $\Lambda$ and $J$ will meet at two distinct points.

Note that $B$, $\widehat{J}$ and $G$ intersect transversely pairwise on $V_p$ and
\begin{equation}\label{TRIGCFSTHMG36E016}
Z.D = dB + 6\Gamma_Z
\end{equation}
in $V$ for some $d\in \ZZ^+$, where $\Gamma_Z\subset Z$ is the proper transform of $\Gamma$.
So Lemma \ref{TRIGCFSLEM001} can be applied here to conclude that there is no such $D$.

Suppose that the general fibers of $Q/C$ are smooth. Then $S_p$ is a disjoint union of two smooth rational curves for $p\in C$ general. Replacing $S$ and $V$ by their normalization, we may assume that both $S$ and $V$ are normal. Then 
\begin{itemize}
\item either $S$ is a disjoint union of two components, each being a ruled surface over $C$,
\item or the Stein factorization of $S\to C$ factors through a smooth curve $\widehat{C}\to C$, which is a degree $2$ morphism.
\end{itemize}

If it is the former, we replace $S$ and $V$ each by one of its two components. Then $V$ is a family of ruled surfaces over $C$ with general fibers $\FF_0 = \PP^1\times \PP^1$ and special fibers $\FF_2$. We still have the diagram \eqref{TRIGCFSTHMG36E017} and \eqref{TRIGCFSTHMG36E014}-\eqref{TRIGCFSTHMG36E016} continue to hold. Lemma \ref{TRIGCFSLEM001} still applies.

If it is the latter, $V$ is a family of ruled surfaces over $\widehat{C}$ with general fibers $\FF_0 = \PP^1\times \PP^1$ and special fibers $\FF_2$. We replace $X$, $W$ and $C$ by
$\widehat{X} = X\times_C \widehat{C}$, $\widehat{W} = W\times_C \widehat{C}$ and $\widehat{C}$
in the diagram \eqref{TRIGCFSTHMG36E017}:
\begin{equation}\label{TRIGCFSTHMG36E018}
\begin{tikzcd}
& V\ar{d}{\xi} \ar{r}{\phi} & S \ar{d}{\pi}\\
\widehat{X} \ar[dashed]{ur}{\psi}\ar{r}{\tau} & \widehat{W} \ar{r} & \widehat{C}
\end{tikzcd} 
\end{equation}
We let $Z$ and $D\subset V$ be defined accordingly. Then \eqref{TRIGCFSTHMG36E014}-\eqref{TRIGCFSTHMG36E016} continue to hold and Lemma \ref{TRIGCFSLEM001} still applies.

In conclusion, the existence of $\Lambda\in |\OO_W(1)|$ with $\tau^* \Lambda = 6\Gamma$ implies that $Y_p\ne 2J$. So $Y_p = Q_p\cap R_p$ is an integral curve for all $p\in C$. 
\end{proof}

\begin{proof}[Proof of Theorem \ref{TRIGCFSTHMG36} in the case \eqref{TRIGCFSTHMG36E020}]
By Proposition \ref{TRIGCFSPROP002}, $Y$ is a local complete intersection smooth in codimension one.
So Serre's criterion tells us that $Y$ is normal. And since $\tau: X\to Y$ is finite along $\Gamma$, it is a local isomorphism along $\Gamma$. Therefore, $Y_p$ is smooth at the point $Y_p\cap \Lambda$ for all $p\in C$, where $\Lambda\in |\OO_W(1)| \cong |\OO_X(6\Gamma)|$ is the unique section. And since $Y_p$ is a complete intersection of a quadric and a cubic in $W_p\cong \PP^3$,
\begin{equation}\label{TRIGCFSTHMG36E022}
\omega_{Y_p} = \OO_W(1)\otimes \OO_{Y_p} = \OO_{Y_p}(6q)
\end{equation}
for all $p\in C$, where $q = Y_p\cap \Lambda$. Hence
$$
H^0(\OO_{X_p}(6\Gamma))\cong H^0(\OO_{W_p}(1)) \cong H^0(\OO_{Y_p}(6q))
$$
and we have the identification
\begin{equation}\label{TRIGCFSTHMG36E024}
\begin{aligned}
H^0(\OO_{X_p}(a\Gamma)) &= \big\{A\in H^0(\OO_{W_p}(1)): (A.Y_p)_q \ge 6-a\big\}\\
&= H^0(\OO_{Y_p}(aq))
\end{aligned}
\end{equation}
for all $p\in C$ and $0\le a \le 6$.

Let us consider the rational map
\begin{equation}\label{TRIGCFSTHMG36E025}
\begin{tikzcd}
X\ar[dashed]{r}{\alpha} & S = \PP (f_* \OO_X(a\Gamma))^\vee.
\end{tikzcd}
\end{equation}
We claim that $\alpha$ is a regular generically finite map for either $a = 3$ or $a=4$. This follows from a similar argument for Corollary \ref{TRIGCFSCORPCM}.

First, let us prove
\begin{equation}\label{TRIGCFSTHMG36E023}
h^0(\OO_{X_p}(2\Gamma)) = 1
\end{equation}
for all $p\in C$. Otherwise, by \eqref{TRIGCFSTHMG36E024}, there exists a hyperplane $A\ne \Lambda$ in $W_p$ such that $(A.Y_p)_q \ge 4$ for $q= Y\cap \Lambda$.

If $Q_p$ is smooth, $\Lambda \cap Q_p$ is a curve of type $(1,1)$ in $Q_p \cong \PP^1\times \PP^1$. If $G = \Lambda . Q_p$ is not integral, then
$G = G_1 + G_2$ with $G_1$ and $G_2$ meeting transversely at $q$. Since $G.Y_p = 6q$ on $Q_p$, $G_1.Y_p = 3q$ and $G_2.Y_p = 3q$. This is impossible since $Y_p$ is smooth at $q$ and $G_1$ and $G_2$ meeting transversely at $q$. So $G$ is integral. Hence $A, \Lambda$ and $Q_p$ meet properly in $W_p$. And since $(A.Y_p)_q \ge 4$ and $(\Lambda.Y_p)_q = 6$,
$$
(A.\Lambda.Q_p)_q \ge 4
$$
which is a contradiction since $A\Lambda Q_p = 2$.

Suppose that $Q_p$ is not smooth. Since it is integral, it is a cone over a smooth conic curve. Let $o$ be its vertex. Since $Y_p$ is a Cartier divisor on $Q_p$ and $Y_p$ is smooth at $q$, we see that $q\ne o$.

If $\Lambda$ does not pass through $o$, then $\Lambda\cap Q_p$ is an integral curve. So $A, \Lambda$ and $Q_p$ meet properly in $W_p$. Then the same argument as before shows that
$$
(A.\Lambda.Q_p)_q \ge 4
$$
which is a contradiction since $A\Lambda Q_p = 2$.

Suppose that $\Lambda$ passes through $o$. Then $G = \Lambda.Q_p = L_1 + L_2$ is a union of two lines $L_1$ and $L_2$ passing through $o$. Since $L_1\cap Y_p = L_2\cap Y_p = q\ne o$, $L_1 = L_2 = L$. So $G = 2L$ and $L.Y_p = 3q$ on $Q_p$.

Let $F = A.Q_p$. If $A$ passes through $o$, then $F = L + L'$ for a line $L'\ne L$. But $q\not\in L'$ and hence $(F.Y_p)_q = (L.Y_p)_q = 3$, which contradicts the fact that $(A.Y_p)_q \ge 4$. So $A$ does not pass through $o$ and $L\not\subset F$. Hence $F$ and $L$ meet properly on $Q_p$. And since $(L.Y_p)_q = 3$, $(F.Y_p)_q \ge 4$ and $Y_p$ is smooth at $q$, we must have
$$
(F.L)_q \ge 3
$$
on $Q_p$, which is a contradiction since $FL = 2$ on $Q_p$.

This proves \eqref{TRIGCFSTHMG36E023}. It follows that
\begin{equation}\label{TRIGCFSTHMG36E026}
h^0(\OO_{X_p}(3\Gamma)) = h^0(\OO_{Y_p}(3q)) \le 2 \hspace{12pt} \text{for all } p\in C.
\end{equation}

If $f_* \OO_X(3\Gamma)$ has rank $2$, then \eqref{TRIGCFSTHMG36E023} implies that the map $\alpha$ in \eqref{TRIGCFSTHMG36E025} is a regular generically finite map for $a = 3$.

Suppose that $f_* \OO_X(3\Gamma)$ has rank $1$. Then
\begin{equation}\label{TRIGCFSTHMG36E027}
h^0(\OO_{X_p}(3\Gamma)) = h^0(\OO_{Y_p}(3q)) = 1 \hspace{12pt}
\text{for } p\in C \text{ general}.
\end{equation}

Combining \eqref{TRIGCFSTHMG36E026} and \eqref{TRIGCFSTHMG36E027}, we conclude 
\begin{equation}\label{TRIGCFSTHMG36E028}
h^0(\OO_{X_p}(3\Gamma)) = h^0(\OO_{Y_p}(3q)) = 1 \hspace{12pt} \text{for all } p\in C,
\end{equation}
by Harris' theorem on theta characteristic.
It follows that the map $\alpha$ in \eqref{TRIGCFSTHMG36E025} is a regular generically finite map for $a = 4$.

In conclusion, $\alpha$ is regular and $S$ is a ruled surface for either $a=3$ or $a=4$. Then by the same argument as before,
$$
\Gamma^2 \le \dfrac{1}a \left(-d + \dfrac{\deg K_C}2\right).
$$
Therefore,
\begin{equation}\label{TRIGCFSTHMG36E029}
\begin{aligned}
K_X^2 &= K_X(f^* D + 6\Gamma + V) \ge K_X(f^* D + 6\Gamma)\\
&= 6d + 6K_X \Gamma = 6d + 6\deg K_C - 6\Gamma^2\\
&\ge 6d + 6\deg K_C - \dfrac{6}{a}\left(-d + \dfrac{\deg K_C}2\right)\\
&= \dfrac{6a+6}a d + \dfrac{6a-3}a \deg K_C \ge \dfrac{15}{2} d + \dfrac{21}{4} \deg K_C.
\end{aligned}
\end{equation}
\end{proof}

So we have two lower bounds for $K_X^2$:
In the case \eqref{TRIGCFSTHMG36E020}, we have \eqref{TRIGCFSTHMG36E029}; otherwise, we have
\eqref{TRIGCFSTHMG36E021}. In conclusion, we obtain \eqref{TRIGCFSE408} for $g=4$.


\begin{thebibliography}{Miy84}

\bibitem[Bea79]{Beauville:1979}
Arnaud Beauville.
\newblock L'application canonique pour les surfaces de type g\'{e}n\'{e}ral.
\newblock {\em Invent. Math.}, 55(2):121--140, 1979.

\bibitem[Che17]{CHEN20171033}
Xi~Chen.
\newblock Xiao's conjecture on canonically fibered surfaces.
\newblock {\em Advances in Mathematics}, 322:1033--1084, 2017.

\bibitem[EV92]{E-V}
H\'el\`ene Esnault and Eckart Viehweg.
\newblock {\em Lectures on Vanishing Theorems}.
\newblock DMV Seminar Vol 20. Springer-Verlag, Berlin-New York, 1992.

\bibitem[Har82]{38d847b1-588f-3bb6-91c3-8f7be692ac65}
Joe Harris.
\newblock Theta-characteristics on algebraic curves.
\newblock {\em Transactions of the American Mathematical Society},
  271(2):611--638, 1982.

\bibitem[Lu20]{Lu:2020:Canonically:fibered:gen:type:surfaces}
Xin Lu.
\newblock Canonically fibered surfaces of general type.
\newblock {\em J. Inst. Math. Jussieu}, 19(1):209--229, 2020.

\bibitem[Miy84]{LOGMY}
Yoichi Miyaoka.
\newblock The minimal number of quotient singularities on surfaces with given
  numerical invariants.
\newblock {\em Math. Ann.}, 268:159--171, 1984.

\bibitem[SD73]{Saint-Donat1973}
B.~Saint-Donat.
\newblock On petri's analysis of the linear system of quadrics through a
  canonical curve.
\newblock {\em Mathematische Annalen}, 206(2):157--175, Jun 1973.

\bibitem[Sun94]{Sun1994}
Xiaotao Sun.
\newblock On canonical fibrations of algebraic surfaces.
\newblock {\em manuscripta mathematica}, 83(1):161--169, Dec 1994.

\bibitem[Xia85]{Xiao1985}
G.~Xiao.
\newblock L'irrégularité des surfaces de type général dont le système
  canonique est composé d'un pinceau.
\newblock {\em Compositio Mathematica}, 56(2):251--257, 1985.

\bibitem[Xia88]{problem}
G.~Xiao.
\newblock Problem list in birational geometry of algebraic varieties: open
  problems.
\newblock {\em in: XXIII International Symposium, Division of Mathematics, the
  Taniguchi Foundation.}, pages pp. 36--40, 1988.

\end{thebibliography}

\end{document}